\documentclass[11pt]{amsart}
\usepackage{a4wide}
\usepackage{amsmath}
\usepackage[utf8]{inputenc}
\usepackage{amssymb}
\usepackage{amsopn}
\usepackage{epsfig}
\usepackage{amsfonts}
\usepackage{latexsym}
\usepackage{graphicx}
\usepackage{calrsfs}
\usepackage{tikz}
\usepackage{enumerate}
\usepackage{dsfont}
\allowdisplaybreaks[1]

 \usetikzlibrary{arrows,automata}
\DeclareMathAlphabet{\pazocal}{OMS}{zplm}{m}{n}
\newtheorem{theorem}{Theorem}[section]
\newtheorem{lemma}[theorem]{Lemma}
\newtheorem{proposition}[theorem]{Proposition}

\newtheorem{claim}{Claim}

\theoremstyle{definition}
\newtheorem{definition}[theorem]{Definition}

\theoremstyle{remark}
\newtheorem{remark}[theorem]{Remark}

\numberwithin{equation}{section}

\newcommand{\R}{\ensuremath{\mathbb{R}}}
\newcommand{\N}{\ensuremath{\mathbb{N}}}

\newcommand{\set}[1]{\left\{#1\right\}}
\newcommand{\fr}[1]{\langle #1\rangle }
\newcommand{\fl}[1]{\lfloor #1\rfloor }
\newcommand{\la}{\lambda}

\newcommand{\f}{\infty}

\newcommand{\al}{\alpha}

\newcommand{\lf}{\lfloor}
\newcommand{\rf}{\rfloor}

\begin{document}
\title{Approximation properties of the intermediate $\beta$-expansions}

\author[K. Dajani]{Karma Dajani}
\address[K. Dajani]{Department of Mathematics, Utrecht University,  P.O. Box 80010, 3508TA Utrecht, The Netherlands.}
\email{k.dajani1@uu.nl}

\author[Y. Huang]{Yan Huang}
\address[Y. Huang]
{School of Mathematics and Statistics, Wuhan University, Wuhan 430072, Hubei, People's Republic of China.}
\email{yanhuangyh@126.com}

%\date{\today}
%\dedicatory{}

%\begin{frontmatter}


\begin{abstract}
Given $\beta>1$ and $\alpha\in[0,1)$, let $T_{\beta, \alpha}(x)=\beta x+\alpha\pmod 1$. Then under the map $T_{\beta,\alpha}$ each $x\in[0,1]$ has an \emph{intermediate $\beta$-expansion} of the form $x=\sum_{i=1}^\infty\frac{c_i-\alpha}{\beta^i}$  {with each $c_i\in\{0,1,\ldots,\lf \beta+\alpha\rf\}$}.
In this paper we study the approximation properties of $T_{\beta,\alpha}$ by considering the expected value $M_\beta(\alpha)$ of the  \emph{normalized errors} $(\theta_{\beta,\alpha}^n(x))_{n\geq 1}$, where
$$\theta_{\beta,\alpha}^n(x):=\beta^n\left(x-\sum_{i=1}^n\frac{c_i-\alpha}{\beta^i}\right),\quad n\in\mathbb{N}.$$
We prove that  $M_\beta(\cdot)$ is continuous on   $[0,1)$.
As a result,  
$\mathcal{M_\beta}:=\{M_\beta(\alpha):\alpha\in[0,1)\}$ is a closed interval.
In particular, if $\beta$ is a multinacci number, the map $T_{\beta,\alpha}$ has matching for Lebesgue almost every $\alpha\in[0,1)$, and then 
 $M_\beta(\cdot)$ is   locally linear almost everywhere on $[0,1)$. 
 \end{abstract}
\maketitle

\section{Introduction}\label{sec: Introduction}
The study of non-integer base expansions were initiated by R\'enyi \cite{Renyi-1957} and Parry \cite{Parry-1960, Parry-1964}.
Since then, expansions in non-integer bases have received much attention and have connections with many areas of mathematics such as ergodic theory, fractal geometry, symbolic dynamics, and number theory.
They have gained momentum after the surprising discovery by Erd\H os et al.~\cite{Erdos_Horvath_Joo_1991, Erdos_Joo_1992} that for any $k\in\N\cup\set{\aleph_0}$ there exist $\beta\in(1,2]$ and $x\in[0, 1/(\beta-1)]$ such that $x$ has precisely $k$ different $\beta$-expansions. This phenomenon is completely different from the integer base expansions where each $x$ has a unique expansion, except for countably many $x$ having two expansions. Typically each point in $[0,1]$ has uncountably many representations \cite{Sidorov-2003}. The largest one in the lexicographic order is called the \emph{greedy expansion} and the smallest one  is called the \emph{lazy expansion}. An interesting feature of these extreme cases is that they can be generated dynamically by iterating the so-called \emph{greedy $\beta$-transformation} and \emph{lazy $\beta$-transformation}, respectively. It is natural to  ask which one has better approximation properties.
It has been proven that “on average”
the greedy convergents  $\left(\sum_{i=1}^{n}\frac{a_i}{\beta^i}\right)_{n\geq 1}$ approximate $x$ better than the lazy convergents
$\left(\sum_{i=1}^{n}\frac{b_i}{\beta^i}\right)_{n\geq 1}$ for almost every $x\in [0,1]$ in \cite{Dajani-Cor-2002}. So, what about those expansions between greedy and lazy which are generated by iterating an appropriate transformation? There are some results that focus on the Diophantine approximation properties of $\beta$-expansions, see \cite{Lv-2013,Ma-Wang-2015,Wang-Li-2020}. Baker also studied the approximation properties of $\beta$-expansions in \cite{Baker-2015,Baker-2018}. 

In this paper, we address this question through an analysis of approximation properties in \emph{intermediate $\beta$-expansions}. 
These expansions were first introduced by R\'{e}nyi in 1957 \cite{Renyi-1957} and can be dynamically generated by iterating the \emph{intermediate $\beta$-transformations}. 
We begin by presenting the formal definition and fundamental properties of these transformations.

Given $\beta > 1$ and $\alpha \in [0,1)$, the intermediate $\beta$-transformation $T_{\beta,\alpha}: [0,1] \to [0,1]$ is defined by
\begin{equation}\label{eq:map-T}
T_{\beta,\alpha}(x) =\beta x +\alpha \pmod 1=\beta x+\al-\lf\beta x +\al\rf.
\end{equation}
For a real number $r$ we denote by $\fr{r}$ and $\fl{r}$ its fractional   and integer parts, respectively. Then by (\ref{eq:map-T}) it follows that for $\alpha\in[0,1-\fr{\beta})$,  
\begin{equation}\label{eq:def-T-leq-1}
\begin{split}
T_{\beta,\alpha}(x)=\left\{
\begin{array}{ll}
\beta x+\alpha,  \quad &\textrm{if}\quad x\in[0,\frac{1-\alpha}{\beta});\\
\beta x+\alpha-i,\quad &\textrm{if}\quad  x\in[\frac{i-\alpha}{\beta},\frac{i+1-\alpha}{\beta}),\quad i\in\{1,\ldots , \lf \beta\rf-1\};\\
\beta x+\alpha-\lf \beta\rf,\quad &\textrm{if}\quad  x\in[\frac{\lf \beta\rf-\alpha}{\beta},1];\\
\end{array}
\right.
\end{split}
\end{equation}
and for $\alpha\in[1-\fr{\beta},1)$ we have
\begin{equation}\label{eq:def-T-geq-1}
\begin{split}
T_{\beta,\alpha}(x)=\left\{
\begin{array}{ll}
\beta x+\alpha, \quad &\textrm{if}\quad x\in[0,\frac{1-\alpha}{\beta});\\
\beta x+\alpha-i,\quad &\textrm{if}\quad x\in[\frac{i-\alpha}{\beta},\frac{i+1-\alpha}{\beta}),\quad i\in\{1,\ldots , \lf \beta\rf\};\\
\beta x+\alpha-\lf \beta\rf-1,\quad &\textrm{if}\quad x\in[\frac{\lf \beta\rf+1-\alpha}{\beta},1].\\
\end{array}
\right.
\end{split}
\end{equation}
Clearly, the map $T_{\beta,\alpha}$ has $\lf\beta\rf+1$ branches if $\alpha\in[0,1-\fr{\beta})$, and has $\lf\beta\rf+2$ branches if $\alpha\in[1-\fr{\beta}, 1)$. See Figure \ref{f:T-beta-alpha} for an illustration.

\begin{figure}[h]
\begin{tikzpicture}[scale=3.1]
\draw(-.01,0)node[below]{\footnotesize $0$}--(5/9-5/90,0)node[below]{\footnotesize $\frac{1-\alpha}{\beta}$}--(1,0)node[below]{$1$};
\draw(0,-.01)--(0,1)node[left]{\footnotesize $1$};
\draw[thick, green!60!black] %map
(0,0.1)--(5/9-5/90,1)(5/9-5/90,0)--(1,9/10);
\draw[dotted](0,1)--(1,1)--(1,0)(5/9-5/90,0)--(5/9-5/90,1);
\end{tikzpicture}
\quad
\begin{tikzpicture}[scale=3.1]
\draw(-.01,0)node[below]{\footnotesize $0$}--(5/9-15/90,0)node[below]{\footnotesize $\frac{1-\alpha}{\beta}$}--(10/9-15/90,0)node[below]{\footnotesize $\frac{2-\alpha}{\beta}$}--(1,0);
\draw(0,-.01)--(0,1)node[left]{\footnotesize $1$};
\draw[thick, green!60!black] %map
(0,0.3)--(5/9-15/90,1)(5/9-15/90,0)--(10/9-15/90,1)(10/9-15/90,0)--(1,0.1);
\draw[dotted](0,1)--(1,1)--(1,0)(5/9-15/90,0)--(5/9-15/90,1)(10/9-15/90,0)--(10/9-15/90,1);
\end{tikzpicture}
\quad
{\caption{Fix $\beta=1.8$. The left figure shows the image of the transformation $T_{\beta,\alpha}$ for $\alpha=0.1 \in [0,1-\langle\beta\rangle)$, and the right figure for $\alpha=0.3 \in [1-\langle\beta\rangle,1)$.}}
\label{f:T-beta-alpha}
\end{figure}

It is well known that, in general, $T_{\beta,\alpha}$ does not  preserve Lebesgue measure and  {is} not a subshift of finite type with mixing properties.
This causes difficulties in studying metrical questions related to $\beta$-expansions. We mention few results regarding these transformations.
Hofbauer proved that for any $\alpha\in[0,1)$ the topological entropy of $T_{\beta,\alpha}$ is $\log\beta$, and $T_{\beta,\alpha}$ has a unique invariant probability measure of maximal entropy in \cite{Hofbauer-1980}. This maximal measure is absolutely continuous with
respect to the Lebesgue measure, and its support is a finite union of intervals (Hofbauer \cite{Hofbauer-1978, Hofbauer-1981}). Its density function was given by Parry in \cite{Parry-1960}, and was later proved to be non-negative by Halfin in \cite{Halfin-1975}.
Concerning the ergodic properties of $T_{\beta,\alpha}$, Wilkinson \cite{Wilkinson-1974, Wilkinson-1974-1975}
showed that $T_{\beta,\alpha}$ is weak-Bernoulli with respect to the maximal measure when $\beta > 2$, and
this was extended first by Parry \cite{Parry-1978} to $\beta > \sqrt 2$.   However, this is not true for $\beta\in(1,\sqrt{2}]$. Palmer \cite{Flatto-Lagarias-1996} characterized
all pairs {$(\beta,\alpha)\in (1,\sqrt 2]\times [0,1)$} in which $T_{\beta,\alpha}$   is weak-Bernoulli, and proved that when the pair  $(\beta,\alpha)$ is not in the ``bubble", the map $T_{\beta,\alpha}$ is weak-Bernoulli. Importantly,   the map  $T_{\beta,\alpha}$ is ergodic for any $\beta>\sqrt 2$ and $\alpha\in[0,1)$ (cf.~\cite{Cabane-1979}).

 The following Theorem describes the density function of the unique $T_{\beta,\alpha}$-invariant probability measure of maximal entropy (see Parry \cite{Parry-1964}).
 Throughout this article, we denote the Lebesgue measure on $\R$ by $\lambda$.
\begin{theorem}[Parry, 1964] \label{thm:T-ergodic-density}
Given $\beta>1$ and $ \alpha\in[0,1)$, let $T_{\beta,\alpha}(x)=\beta x +\alpha \pmod 1$.
Then  there exists a unique $T_{\beta,\alpha}$-invariant probability measure $\nu_{\beta,\alpha}$ on $[0,1]$  such that
\[\nu_{\beta,\alpha}(E)=\int_E g_{\beta,\alpha}(x) \mathrm{d}\lambda(x)\]
  for all Borel set $E\subset[0,1]$, where the density function $g_{\beta,\alpha}$ is given by
\begin{equation}\label{eq:density-normal}
g_{\beta,\alpha}(x)=\frac{\tilde g_{\beta,\alpha}(x)}{\int_{[0,1]}\tilde g_{\beta,\alpha}d\lambda}\quad  \text{with}\quad
\tilde g_{\beta,\alpha}(x)=\sum_{x<T_{\beta,\alpha}^n(1)}\frac{1}{\beta^n}-\sum_{x<T_{\beta,\alpha}^n(0)}\frac{1}{\beta^n}.
\end{equation}
Furthermore, 
\begin{enumerate}[{\rm(i)}]
  \item $\tilde g_{\beta,\alpha}(x)\geq 0$ for $\lambda$ almost every $x\in[0,1]$.
  \item If $\beta> \sqrt 2$ and $\alpha\in[0,1)$, then   $T_{\beta,\alpha}$ is ergodic with respect to $\nu_{\beta,\alpha}$.
\end{enumerate}
\end{theorem}

Given $\beta>1$ and $\alpha\in[0,1)$, recall that every $x\in[0,1]$ admits an intermediate $\beta$-expansion generated by $T_{\beta,\alpha}$. 
We now formally describe the dynamical mechanism through which this transformation produces expansions. For $x\in[0,1]$ and $n\ge 0$ let 
\begin{equation}\label{eq:digit-T-beta-alpha}
c_{n+1}:=\fl{\beta T_{\beta,\alpha}^n(x)+\alpha}.
\end{equation}
Then by (\ref{eq:map-T}) it yields that  $T_{\beta, \alpha}(x)=\beta x+\alpha-c_1$, which gives
$
x=\frac{T_{\beta,\alpha}(x)}{\beta}+\frac{c_1-\alpha}{\beta}.
$
Note that $T_{\beta,\alpha}(x)\in[0,1]$. Replacing $x$ by $T_{\beta,\alpha}(x)$ in the above argument it follows that 
$
x=\frac{T_{\beta,\alpha}^2(x)}{\beta^2}+\frac{c_2-\alpha}{\beta^2}+\frac{c_1-\alpha}{\beta}.
$
Iterating this argument, we obtain the  intermediate $\beta$-expansion of  $x$ by
\begin{equation}\label{eq:def-beta-expansion}
x=\sum_{i\geq1}\frac{c_i-\alpha}{\beta^i},\quad c_i\in\{0,1,\ldots ,\lf \beta+\alpha \rf\}.
\end{equation}
The infinite sequence $(c_i)$ is also called the intermediate $\beta$-expansion of $x$ with parameter $\alpha$.

In \cite{optimal-Dajani-Martijn-Vilmos-Paola-2012} Dajani et al.~introduced the \emph{normalized errors} of $x$ with respect to $T_{\beta,\alpha}$, which are defined by 
\begin{equation*}
{\theta_{\beta,\alpha}^n(x)}:=T_{\beta,\alpha}^n(x)=\beta^n\left(x-\sum_{i=1}^n\frac{c_i-\alpha}{\beta^i}\right),\quad n\in\mathbb{N}.
\end{equation*}
Note that the sequence  $\left(\theta_{\beta,\alpha}^n(x)\right)_{n\geq 1}$  reveals the approximation efficiency of the convergents $\left(\sum_{i=1}^n\frac{c_i-\alpha}{\beta^i}\right)_{n\geq 1}$ to the point $x$. If $\beta>\sqrt{2}$ and $\alpha\in[0,1)$, then by Theorem \ref{thm:T-ergodic-density} and a straightforward application of Birkhoff's ergodic theorem it follows that for $\nu_{\beta,\alpha}$ almost every $x\in[0,1]$ we have
\begin{equation}\label{eq:def-M}
%\lim_{n\to\infty}\frac{1}{n}\sum_{i=0}^{n-1}\theta_{\beta,\alpha}^i(x)=
\lim_{n\to\infty}\frac{1}{n}\sum_{i=0}^{n-1}\theta_{\beta,\alpha}^i(x)=\int_{[0,1)}xg_{\beta,\alpha}(x)d\lambda(x)=:M_{\beta}(\alpha).
\end{equation}
The smaller $M_{\beta}(\alpha)$ is, the better  approximation efficiency of   $T_{\beta,\alpha}$ is. So, the quantity $M_{\beta}(\alpha)$ characterizes the approximation efficiency of $T_{\beta,\alpha}$. Observe by (\ref{eq:map-T}) that for 
 $\alpha=0$, the map $T_{\beta,0}$ yields the greedy $\beta$-expansion of $x$; while for $\alpha=1-\fr{\beta}$ the map $T_{\beta,1-\fr{\beta}}$ gives the lazy $\beta$-expansion of $x$. Dajani et al.~\cite{optimal-Dajani-Martijn-Vilmos-Paola-2012} showed that $M_{\beta}(0)<M_{\beta}(1-\fr{\beta})$ for any $\beta>1$. This means in general the greedy $\beta$-transformation has better approximation properties than the lazy $\beta$-transformation.  
In general, what can we say about the approximation efficiency of  $T_{\beta,\alpha}$ for a general $\alpha\in[0,1)$? This will be our main object to study in this paper.

Our first   result states that $M_\beta$ is continuous.
\begin{theorem}\label{thm:main-1-continuous}
Let  $\beta>\sqrt 2$. Then $M_{\beta}$ is continuous on $[0,1)$. Furthermore, $M_\beta(\alpha)+M_\beta(1-\fr{\beta+\alpha})=1$ for any $\alpha\in[0,1)$.
\end{theorem}
\begin{remark}
An adaption of the proof of Theorem \ref{thm:main-1-continuous} can show that the map $\beta\mapsto M_{\beta}(\alpha)$ is also continuous on $(\sqrt{2}, \infty)$ for any $\alpha\in[0,1)$.
\end{remark}
In the following we fix $\beta>\sqrt{2}$, and set
\begin{equation*}\label{eq:def-set-M}
\mathcal{M}_\beta:=\{M_\beta(\alpha):\alpha\in[0,1)\}.
\end{equation*}
Note by Theorem \ref{thm:main-1-continuous} that $M_\beta(\alpha)+M_\beta(1-\fr{\beta+\alpha})=1$ for any $\alpha\in[0,1)$. Then
\[\left\{
\begin{array}{cc}
   M_\beta(\alpha)+M_\beta(1-\fr{\beta}-\alpha)=1&\textrm{if}\quad \alpha\in[0,1-\fr{\beta}); \\
   M_\beta(\alpha)+M_\beta(2-\fr{\beta}-\alpha)=1&\textrm{if}\quad \alpha\in[1-\fr{\beta}, 1).
\end{array}\right.
\]
This, together with the continuity property of $M_\beta$ in Theorem \ref{thm:main-1-continuous}, implies that  $\mathcal{M}_\beta$ is a closed interval in $[0,1)$ symmetric about the line $y=1/2$  (see Figure \ref{f:no-mono}  for an example).

Note that $M_{\beta}(\alpha)$ depends on the density function $g_{\beta,\alpha}$ which is in general  an infinite series (see Theorem \ref{thm:T-ergodic-density}). This leads to difficulties in studying extreme values of $\mathcal{M}_\beta$.
  \begin{figure}[h!]
  \centering
   \includegraphics[width=6.5cm]{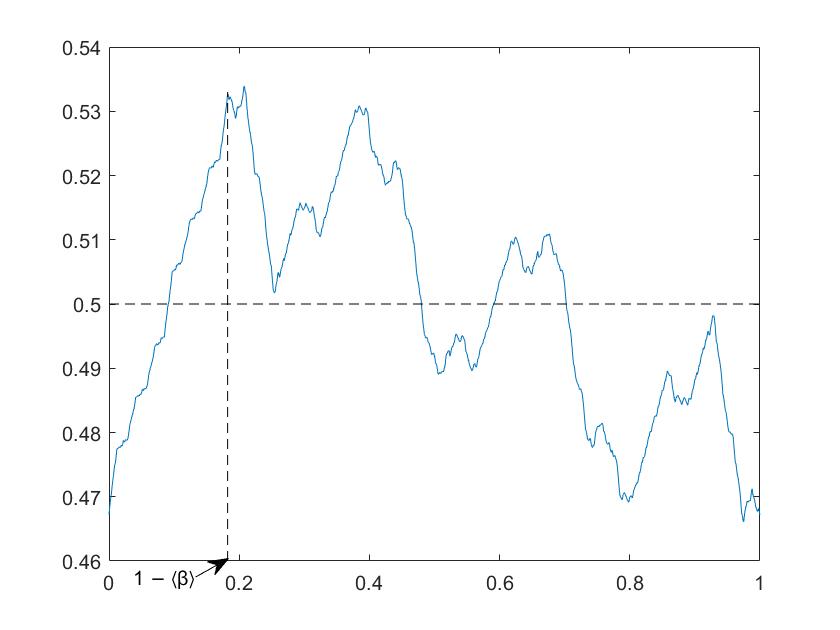}\quad \includegraphics[width=6.5cm]{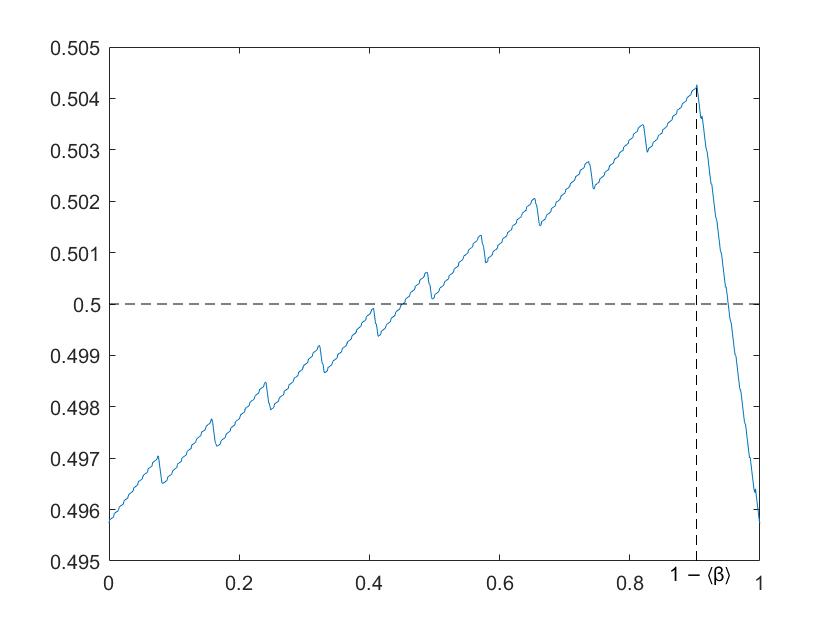}\\
  \caption{The  graphs of $M_{\beta}(\alpha)$ for $\alpha\in[0,1)$ with $\beta\approx 2.8177$ (left) and $\beta\approx 11.0976 $ (right). }\label{f:no-mono}
\end{figure}
However, $g_{\beta,\alpha}$ becomes a finite sum if the orbits of $0$ and $1$ meet after a finite number of iterations, this phenomena is called \emph{matching}.
Especially in the case of $\alpha$-continued fraction maps, the concept of matching has been proven to be very useful (cf.~\cite{Carminati-Marmi-Profeti-Tiozz-2010,Carminati-Tiozzo-2012,Kalle-Langeveld-Maggioni-Munday-2020,Nakada-Natsui-2008}).
For piecewise linear transformations, prevalent matching appears to be rare: for instance it can occur for $T_{\beta,\alpha}$ only if the slope $\beta$ is an algebraic integer. In \cite{Bruin-Carminati-Kalle-2017}  it is shown that matching occurs for all quadratic Pisot slopes and a set of translation parameters whose complement has Hausdorff dimension smaller than $1$. Later Bruin and Keszthelyi extended this result in \cite{Bruin-Keszthelyi-2022}. Sun et al.~\cite{Sun-Li-Ding-2023}  gave some dense fibre results of matching for fixed $\beta\in(1,2)$ and $\alpha\in[0,1-\fr{\beta})$.
We point out that all of these  results for $T_{\beta,\alpha}$ focus on   $\beta\in(1,2)$ or $\alpha\in[0,1-\fr{\beta})$.

In the following we assume $\beta$ belongs to a special class of algebraic integers.      Given  $ q , m\in\N$, we define the \emph{$( q , m)$-multinacci number} $\beta_{ q ,m}\in( q , q +1)$ to be the real root of $$ x^m= q  x^{m-1}+ q  x^{m-2}+\ldots + q .$$
  Our second result shows that    matching occurs  for $T_{\beta_{ q ,m},\alpha}$ for typical $\alpha\in[0,1)$.
%\begin{theorem}\label{thm:main-2-linear}
% Let $\beta=\beta_{ q ,m}$ with $m, q \in\N$. Then $T_{\beta_{ q ,m},\alpha}$ has matching for Lebesgue almost every $\alpha\in[0,1)$.
%\end{theorem}
\begin{theorem}\label{thm:main-2-linear}
For all  $m, q \in\N$, the transformation $T_{\beta_{ q ,m},\alpha}$ has matching for Lebesgue almost every $\alpha\in[0,1)$.
\end{theorem}

Given $\beta > \sqrt{2}$, a subinterval $I \subset (0,1]$ is called a \emph{matching interval} if there exists $k \in \mathbb{N}$ such that for every $\alpha \in I$,
\[
T_{\beta,\alpha}^i(0) \neq T_{\beta,\alpha}^i(1) \quad \forall\, 1 \leq i \leq k \quad \text{and} \quad T_{\beta,\alpha}^{k+1}(0) = T_{\beta,\alpha}^{k+1}(1),
\]
where $I$ is maximal with this property. By \eqref{eq:density-normal}, the density function $g_{\beta,\alpha}$ is then a sum of at most $k+1$ terms for all $\alpha \in I$.

We establish that $M_{\beta}$ is linear on each matching interval. Theorem \ref{thm:main-2-linear} implies that matching intervals cover $[0,1)$ up to a Lebesgue null set when $\beta = \beta_{q,m}$. Consequently, $M_{\beta_{q,m}}(\cdot)$ is linear for almost every $\alpha \in [0,1)$. Furthermore, we characterize the monotonicity of $M_{\beta_{q,m}}$ on each matching interval (See details in Section \ref{sec:matching-2}).
In particular, $[0, 1-\fr{\beta_{ q ,m}})$ is a matching interval where $M_{\beta_{q,m}}$ is strictly increasing (see Section \ref{sec:matching-1}). For fixed $q, m \in \mathbb{N}$, write
\[
M_{\beta_{q,m}}(\alpha) := k_{q,m} \alpha + b_{q,m} \quad \text{for} \quad \alpha \in [0, 1 -\fr{\beta_{ q ,m}}).
\]
Our final result states that for $q \in \mathbb{N}$ and any $m \geq 2$,
\[
k_{q,m} < k_{q,m+1} \quad \text{and} \quad b_{q,m} < b_{q,m+1}.
\]
 This implies that $T_{\beta_{q,2},\alpha}$ achieves superior approximation efficiency for $\alpha \in [0, 1  -\fr{\beta_{ q ,m}})$ among $m\geq 2$ (See Figure \ref{Fig:0-1}).
 \begin{theorem}\label{thm:main-3-increasing}
Given  $ q \geq 1$,  both sequences $\{k_{ q ,m}\}_{m\geq 1}$ and $\{b_{ q ,m}\}_{m\geq 1}$ are strictly increasing.
\end{theorem}
\begin{figure}[h!]
  \centering
  \includegraphics[width=8cm]{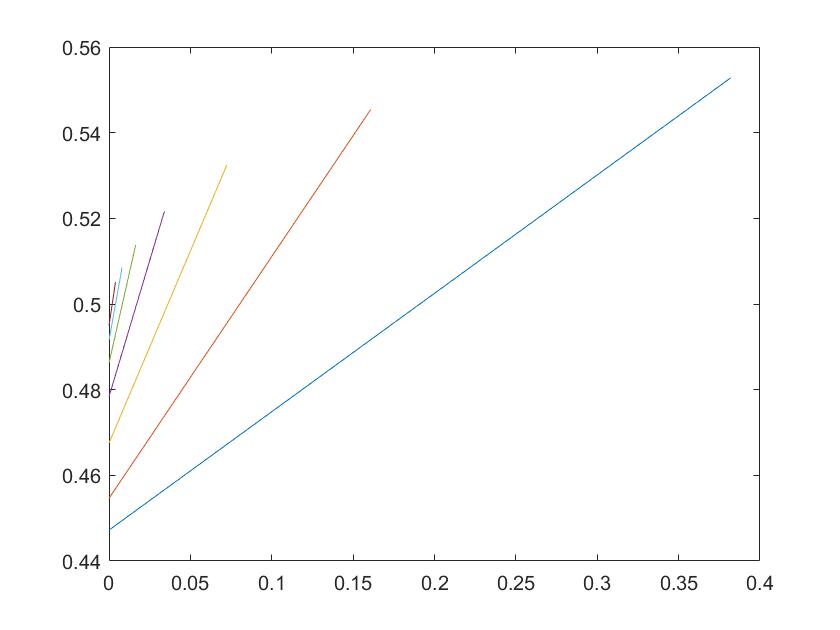}\\
  \caption{The graph of $M_{\beta_{1,m}}(\alpha)=k_{1,m}\alpha+b_{1,m}$ over $\alpha\in[0,1-\fr{\beta_{1,m}})$ for $m=2,\ldots  8$ (from bottom to top).}\label{Fig:0-1}
\end{figure}
 
The rest of the paper is organized as follows. In the next section we study the continuity properties of $M_{\beta}$, and then prove Theorem \ref{thm:main-1-continuous}. The matching results in Theorem \ref{thm:main-2-linear} are  split into two parts, $\alpha\in[0,1-\fr{\beta_{ q ,m}})$ and $\alpha\in[1-\fr{\beta_{ q ,m}},1)$. In Section \ref{sec:matching-1} we show that $M_{\beta_{ q ,m}}$ is linear  and increasing on $[0,1-\fr{\beta_{ q ,m}})$ and prove  Theorem \ref{thm:main-3-increasing}. The matching results for $\alpha\in[1-\fr{\beta_{ q ,m}},1)$ are proven in Section \ref{sec:matching-2}. We give some further questions in Section \ref{sec:question}.

\section{The continuity of $M_{\beta}$}\label{sec:cont}
This section investigates the continuity of the function $M_{\beta}$, culminating in the proof of Theorem \ref{thm:main-1-continuous}. Recall from \eqref{eq:def-M} that 
\[
M_\beta(\alpha) = \int_{[0,1]} x g_{\beta,\alpha}(x)  \mathrm{d}\lambda(x),
\]
where $g_{\beta,\alpha}$ is the density of the invariant probability measure $\nu_{\beta,\alpha}$ for $T_{\beta,\alpha}$ (see Theorem \ref{thm:T-ergodic-density}). So, to prove the continuity of $M_\beta$ it suffices to prove the continuity of the map $\al\to g_{\beta,\al}$ in $L^1(\la)$.  

We first establish right continuity of this mapping.

\begin{lemma}\label{lem:right-continuity-M}
Let $\beta > 1$ and $\alpha \in [0,1)$. For any sequence $(\alpha_k)_{k \geq 1} \subset (\alpha, 1)$ satisfying $\lim_{k \to \infty} \alpha_k = \alpha$, then 
$\lim_{k \to \infty} \| g_{\beta,\alpha_k} - g_{\beta,\alpha} \|_1 = 0.$
\end{lemma}
\begin{proof}

 Since the proof for $\al\in[1-\fr{\beta}, 1)$  is similar, we only focus on $\alpha \in [0, 1-\fr{\beta})$. Fix $\alpha \in \big[0, 1-\fr{\beta})$, and take a sequence $(\alpha_k) \subset (\alpha, 1-\fr{\beta})$ converging to $\alpha$ as $k \to \infty$. By Theorem \ref{thm:T-ergodic-density}, we have 
\[
g_{\beta,\alpha_k}(x) = \frac{\tilde{g}_{\beta,\alpha_k}(x)}{K_{\beta,\alpha_k}} \quad \text{for} \quad x \in [0,1),
\]
where $K_{\beta,\alpha_k} := \int_{[0,1]} \tilde{g}_{\beta,\alpha_k}  \mathrm{d}\lambda \in (0,\infty)$. Then
\begin{align*}
\| g_{\beta,\alpha_k} - g_{\beta,\alpha} \|_1 
&= \int_{[0,1]} \left| \frac{\tilde{g}_{\beta,\alpha_k}}{K_{\beta,\alpha_k}} - \frac{\tilde{g}_{\beta,\alpha}}{K_{\beta,\alpha}} \right|  \mathrm{d}\lambda \\
&\leq \int_{[0,1]} \left| \frac{\tilde{g}_{\beta,\alpha_k}}{K_{\beta,\alpha_k}} - \frac{\tilde{g}_{\beta,\alpha_k}}{K_{\beta,\alpha}} \right|  \mathrm{d}\lambda + \int_{[0,1]} \left| \frac{\tilde{g}_{\beta,\alpha_k}}{K_{\beta,\alpha}} - \frac{\tilde{g}_{\beta,\alpha}}{K_{\beta,\alpha}} \right|  \mathrm{d}\lambda \\
&\leq \frac{\| \tilde{g}_{\beta,\alpha_k} \|_\infty}{K_{\beta,\alpha} K_{\beta,\alpha_k}} \| \tilde{g}_{\beta,\alpha_k} - \tilde{g}_{\beta,\alpha} \|_1 + \frac{1}{K_{\beta,\alpha}} \| \tilde{g}_{\beta,\alpha_k} - \tilde{g}_{\beta,\alpha} \|_1,
\end{align*}
where
$\| \tilde{g}_{\beta,\alpha_k} \|_\infty = \sup_{x \in [0,1]} | \tilde{g}_{\beta,\alpha_k}(x) |$. 
Thus, it suffices to prove
\begin{equation}\label{eq:hat-g-g-k}
\lim_{k \to \infty} \| \tilde{g}_{\beta,\alpha_k} - \tilde{g}_{\beta,\alpha} \|_1 = 0.
\end{equation}

Note that  $\alpha, \alpha_k\in[0,1-\fr{\beta})$ for all $k\in\N$.   Let $(a_i), (b_i)\in\{0,\ldots ,\lf\beta\rf\}^\N$ be the intermediate $\beta$-expansions with parameter $\alpha$ of $0$ and $1$, respectively. Then by (\ref{eq:def-beta-expansion}) we have
\[
0=\sum_{i=1}^\infty\frac{a_i-\alpha}{\beta^i}, \qquad 1=\sum_{i=1}^\infty\frac{b_i-\alpha}{\beta^i}.
\] Similarly, for $k\in\N$ let $(a_{k,i}), (b_{k,i})\in\{0,\ldots ,\lf\beta\rf\}^\N$ be the intermediate $\beta$-expansions with parameter $\alpha_k$ of $0$ and $1$, respectively. Then 
\[
0=\sum_{i=1}^\infty\frac{a_{k,i}-\alpha_k}{\beta^i},\qquad 1=\sum_{i=1}^\infty\frac{b_{k,i}-\alpha_k}{\beta^i}.
\]
 For $k\in\N$ set \begin{equation}\label{eq:def-N_k}
 N_k:=\max\{i\in\mathbb{N}:a_{k,1}\ldots  a_{k,i}=a_1\ldots  a_i\quad\text{and}\quad b_{k,1}\ldots  b_{k,i}= b_1\ldots  b_i\}.
 \end{equation}
Note that $N_k\ge 1$ since $a_{k,1}=0 =a_1$ and $b_{k,1}=\lf\beta\rf= b_1$ for all $\alpha_k\in(\alpha,1-\fr{\beta})$. Observe that $\al_k>\alpha$ and $\lim_{k\to\f}\alpha_k=\alpha$. Then we establish the following claim.
\begin{claim}\label{cl:N_k-N}
For every $N \in \mathbb{N}$, there exists $K \in \mathbb{N}$ such that for all $k \geq K$, we have 
\begin{equation}\label{eq:feb4-1}
\begin{split}
0 < \alpha_k - \alpha & < \frac{\beta - 1}{\beta^{N+1} - 1} \cdot \min\left\{\frac{a_i + 1 - \alpha}{\beta} - T_{\beta,\alpha}^{i-1}(0) : 1 \leq i \leq N\right\}, \\
0 < \alpha_k - \alpha & < \frac{\beta - 1}{\beta^{N+1} - 1} \cdot \min\left\{\frac{b_i + 1 - \alpha}{\beta} - T_{\beta,\alpha}^{i-1}(1) : 1 \leq i \leq N\right\},
\end{split}
\end{equation}
and consequently
$ N \leq N_k \leq \log_{\beta}\left(1 + \frac{\beta - 1}{\alpha_k - \alpha}\right).$
\end{claim}

By \eqref{eq:digit-T-beta-alpha}, we have $a_i = \lfloor \beta T_{\beta,\alpha}^{i-1}(0) + \alpha \rfloor > \beta T_{\beta,\alpha}^{i-1}(0) + \alpha - 1$ and $b_i > \beta T_{\beta,\alpha}^{i-1}(1) + \alpha - 1$ for all $ i \geq 1$. This ensures the positivity of the minima of the right terms in \eqref{eq:feb4-1}. Since $\alpha_k > \alpha$ and $\lim_{k \to \infty} \alpha_k = \alpha$, the inequalities in \eqref{eq:feb4-1} hold for all sufficiently large $k$. 

We now establish bounds for $N_k$. First, consider the upper bound. From \eqref{eq:def-beta-expansion}, for any $x \in [0,1]$ we have
\begin{equation}\label{eq:T-beta-alpha}
T_{\beta,\alpha}^n(x) = \beta^n \left( x - \sum_{i=1}^{n} \frac{c_i - \alpha}{\beta^i} \right) \quad \forall~ n \geq 1,
\end{equation}
where $(c_i)_{i=1}^\infty \in \{0,1,\ldots, \lfloor\beta\rfloor\}^\mathbb{N}$ is the intermediate $\beta$-expansion of $x$ with parameter $\alpha$. Combining \eqref{eq:def-N_k} and \eqref{eq:T-beta-alpha} yields
\[
T_{\beta,\alpha_k}^{N_k}(0) - T_{\beta,\alpha}^{N_k}(0) = \beta^{N_k} \sum_{i=1}^{N_k} \frac{\alpha_k - \alpha}{\beta^i} = \frac{\beta^{N_k} - 1}{\beta - 1}(\alpha_k - \alpha).
\]
As $T_{\beta,\alpha_k}^{N_k}(0), T_{\beta,\alpha}^{N_k}(0) \in [0,1]$, it follows that
$
\frac{\beta^{N_k} - 1}{\beta - 1}(\alpha_k - \alpha) \leq 1,
$
implying $$N_k \leq \log_\beta\left(1 + \frac{\beta - 1}{\alpha_k - \alpha}\right).$$

To establish the lower bound of $N_k$, we proceed by induction on $N$. 
The base case $N=1$ holds trivially since $N_k \geq 1$ for all $\alpha_k \in (\alpha,1-\fr{\beta})$. 

Now assume $N \geq 2$, and suppose that $$a_{k,1} \ldots a_{k,n} = a_1 \ldots a_n,\quad b_{k,1} \ldots b_{k,n} = b_1 \ldots b_n\quad \text{for some }n \leq N-1.$$ We will show that $a_{k,n+1} = a_{n+1}$ and $b_{k,n+1} = b_{n+1}$ under assumption \eqref{eq:feb4-1}.
By \eqref{eq:T-beta-alpha} and the condition $\alpha_k > \alpha$, we derive
\begin{align*}
T_{\beta,\alpha_k}^n(0) 
= \beta^n \sum_{i=1}^{n} \frac{\alpha_k - a_i}{\beta^i} 
> \beta^n \sum_{i=1}^{n} \frac{\alpha - a_i}{\beta^i} = T_{\beta,\alpha}^n(0) 
\geq \frac{a_{n+1} - \alpha}{\beta} > \frac{a_{n+1} - \alpha_k}{\beta},
\end{align*}
where the second inequality follows from $a_{n+1} = \lfloor \beta T_{\beta,\alpha}^n(0) + \alpha \rfloor \leq \beta T_{\beta,\alpha}^n(0) + \alpha$. This implies 
$
a_{k,n+1} = \left\lfloor \beta T_{\beta,\alpha_k}^n(0) + \alpha_k \right\rfloor \geq a_{n+1}.
$ 

Conversely, using \eqref{eq:T-beta-alpha} we obtain
\begin{align*}
  T_{\beta,\alpha_k}^n(0) & =\beta^n\sum_{i=1}^{n}\frac{\al_k-a_i}{\beta^i}\\
  &=T_{\beta,\alpha}^n(0)+\beta^n\sum_{i=1}^{n}\frac{\al_k-\al}{\beta^i}=T_{\beta,\alpha}^n(0)+\frac{\beta^n-1}{\beta-1}(\al_k-\al)\\
  &<\frac{a_{n+1}+1-\alpha}{\beta}-\frac{\beta^{N+1}-1}{\beta-1}(\al_k-\al)+\frac{\beta^n-1}{\beta-1}(\al_k-\al)\\
  &<\frac{a_{n+1}+1-\alpha}{\beta}-\frac{\al_k-\al}{\beta}=\frac{a_{n+1}+1-\al_k}{\beta}.
\end{align*}
The first inequality relies on \eqref{eq:feb4-1} which gives
\[
\alpha_k - \alpha < \frac{\beta - 1}{\beta^{N+1} - 1} \left( \frac{a_{n+1} + 1 - \alpha}{\beta} - T_{\beta,\alpha}^n(0) \right).
\]
Consequently, $a_{k,n+1} = \lfloor \beta T_{\beta,\alpha_k}^n(0) + \alpha_k \rfloor < a_{n+1} + 1$, implying that $a_{k,n+1} = a_{n+1}$. 

By identical way applied to $b_{k,1} \ldots b_{k,n} = b_1 \ldots b_n$ and \eqref{eq:feb4-1},
 we obtain $b_{k,n+1} = b_{n+1}$. The induction thus establishes $N_k \geq N$, completing the proof of Claim \ref{cl:N_k-N}.

By \eqref{eq:density-normal}, we have that 
\begin{equation}\label{eq:feb4-2}
\begin{split}
\|\tilde{g}_{\beta,\alpha_k} - \tilde{g}_{\beta,\alpha}\|_1 
&\leq \int_{[0,1]} \left| \sum_{i=0}^{\infty} \frac{\mathds{1}_{[0, T_{\beta,\alpha_k}^i(1))}}{\beta^i} - \sum_{i=0}^{\infty} \frac{\mathds{1}_{[0, T_{\beta,\alpha}^i(1))}}{\beta^i} \right| \mathrm{d}\lambda \\
&\quad + \int_{[0,1]} \left| \sum_{i=0}^{\infty} \frac{\mathds{1}_{[0, T_{\beta,\alpha_k}^i(0))}}{\beta^i} - \sum_{i=0}^{\infty} \frac{\mathds{1}_{[0, T_{\beta,\alpha}^i(0))}}{\beta^i} \right| \mathrm{d}\lambda \\
&\leq \int_{[0,1]} \left| \sum_{i=0}^{N_k} \frac{\mathds{1}_{[0, T_{\beta,\alpha_k}^i(1))}}{\beta^i} - \sum_{i=0}^{N_k} \frac{\mathds{1}_{[0, T_{\beta,\alpha}^i(1))}}{\beta^i} \right| \mathrm{d}\lambda \\
&\quad + \int_{[0,1]} \left| \sum_{i=0}^{N_k} \frac{\mathds{1}_{[0, T_{\beta,\alpha_k}^i(0))}}{\beta^i} - \sum_{i=0}^{N_k} \frac{\mathds{1}_{[0, T_{\beta,\alpha}^i(0))}}{\beta^i} \right| \mathrm{d}\lambda 
 \quad + \sum_{i=N_k+1}^{\infty} \frac{4}{\beta^i}.
\end{split}
\end{equation}
By \eqref{eq:def-N_k} and \eqref{eq:T-beta-alpha}, we have $T_{\beta,\alpha_k}^i(1) > T_{\beta,\alpha}^i(1)$ and $T_{\beta,\alpha_k}^i(0) > T_{\beta,\alpha}^i(0)$ for all $1 \leq i \leq N_k$. Moreover, Claim \ref{cl:N_k-N} implies $N_k \to \infty$ as $k \to \infty$. Combining \eqref{eq:T-beta-alpha}, \eqref{eq:feb4-2}, and Claim \ref{cl:N_k-N} yields
\begin{align*}
\|\tilde{g}_{\beta,\alpha_k} - \tilde{g}_{\beta,\alpha}\|_1 
&\leq \left\| \sum_{i=1}^{N_k} \frac{\mathds{1}_{[T_{\beta,\alpha}^i(1), T_{\beta,\alpha_k}^i(1))}}{\beta^i} \right\|_1 + \left\| \sum_{i=1}^{N_k} \frac{\mathds{1}_{[T_{\beta,\alpha}^i(0), T_{\beta,\alpha_k}^i(0))}}{\beta^i} \right\|_1 + \frac{4}{(\beta-1)\beta^{N_k}} \\
&= \sum_{i=1}^{N_k} \frac{T_{\beta,\alpha_k}^i(1) - T_{\beta,\alpha}^i(1)}{\beta^i} + \sum_{i=1}^{N_k} \frac{T_{\beta,\alpha_k}^i(0) - T_{\beta,\alpha}^i(0)}{\beta^i} + \frac{4}{(\beta-1)\beta^{N_k}} \\
&\leq 2 \sum_{i=1}^{N_k} \frac{1}{\beta^i} \cdot \frac{\beta^i - 1}{\beta - 1} (\alpha_k - \alpha) + \frac{4}{(\beta-1)\beta^{N_k}} \\
&\leq \frac{2}{\beta-1} N_k (\alpha_k - \alpha) + \frac{4}{(\beta-1)\beta^{N_k}} \\
&\leq \frac{2}{\beta-1} (\alpha_k - \alpha) \log_\beta \left(1 + \frac{\beta-1}{\alpha_k - \alpha}\right) + \frac{4}{(\beta-1)\beta^{N_k}} \to 0
\end{align*}
as $k \to \infty$, since $\lim_{k \to \infty} \alpha_k = \alpha$ and $\lim_{k \to \infty} N_k = +\infty$. This establishes \eqref{eq:hat-g-g-k}, completing the proof.
\end{proof}

However, the method from the previous lemma cannot be directly applied to prove the left continuity of the mapping $\beta \mapsto g_{\beta,\alpha}$ in $L^1(\lambda)$. For example, for all $\beta > \sqrt{2}$, take $\alpha = 1 - \fr{\beta}$. The corresponding intermediate $\beta$-expansion of $1$ with $b_1=\lfloor \beta \rfloor +1$. But for all $\alpha_k < \alpha$, the corresponding intermediate $\beta$-expansion of $1$ with $b_{k,1}=\lfloor \beta \rfloor $, which does not satisfy Claim \ref{cl:N_k-N}. 

To address this issue, for $\beta > 1$ and $\alpha \in (0,1]$, we introduce an auxiliary expanding map $\tilde{T}_{\beta,\alpha}: [0,1] \to [0,1]$ defined by
$\tilde{T}_{\beta,\alpha}(x) = \beta x + \alpha - \lceil \beta x + \alpha \rceil + 1$ for all $x \in [0,1]$,
which is left continuous.
 Then, the map admits the following explicit expressions.  
For $\alpha\in(0,1-\fr{\beta}]$ we have  
\begin{equation*}
\begin{split}
\tilde T_{\beta,\alpha}(x)=\left\{
\begin{array}{ll}
\beta x+\alpha,  \quad &\textrm{if}\quad x\in[0,\frac{1-\alpha}{\beta}];\\
\beta x+\alpha-i,\quad &\textrm{if}\quad  x\in(\frac{i-\alpha}{\beta},\frac{i+1-\alpha}{\beta}],\quad i\in\{1,\ldots , \lf \beta\rf-1\};\\
\beta x+\alpha-\lf \beta\rf,\quad &\textrm{if}\quad  x\in(\frac{\lf \beta\rf-\alpha}{\beta},1];\\
\end{array}
\right.
\end{split}
\end{equation*}
and for $\alpha\in(1-\fr{\beta},1]$ we have
\begin{equation*}
\begin{split}
\tilde T_{\beta,\alpha}(x)=\left\{
\begin{array}{ll}
\beta x+\alpha, \quad &\textrm{if}\quad x\in[0,\frac{1-\alpha}{\beta}];\\
\beta x+\alpha-i,\quad &\textrm{if}\quad x\in(\frac{i-\alpha}{\beta},\frac{i+1-\alpha}{\beta}],\quad i\in\{1,\ldots , \lf \beta\rf\};\\
\beta x+\alpha-\lf \beta\rf-1,\quad &\textrm{if}\quad x\in(\frac{\lf \beta\rf+1-\alpha}{\beta},1].\\
\end{array}
\right.
\end{split}
\end{equation*}

Comparison with \eqref{eq:def-T-leq-1} and \eqref{eq:def-T-geq-1} reveals that $\tilde{T}_{\beta,\alpha}$ coincides with $T_{\beta,\alpha}$ except at finitely many points $\left\{ \frac{i - \alpha}{\beta} : i = 1, 2, \ldots, \lfloor\beta\rfloor \right\}$. Adapting the proof of Parry \cite[Theorem 1.6]{Parry-1964} yields the following analogous result for $\tilde{T}_{\beta,\alpha}$.

\begin{theorem}[Parry, 1964] \label{thm:T-ergodic-density-1}
Given $\beta>1$ and $\alpha\in[0,1)$, let $\tilde T_{\beta,\alpha}(x)=\beta x +\alpha -\lceil\beta x+\alpha\rceil+1$ for $x\in[0,1]$. 
There exists a unique $\tilde T_{\beta,\alpha}$-invariant probability measure $\mu_{\beta,\alpha}$ on $[0,1]$ satisfying
\[
\mu_{\beta,\alpha}(E)=\int_E h_{\beta,\alpha}(x)  \mathrm{d}\lambda(x)
\]
for every Borel set $E\subset[0,1]$, where the density function $h_{\beta,\alpha}$ is given by
\begin{equation}\label{eq:density-normal-1}
h_{\beta,\alpha}(x)=\frac{\tilde h_{\beta,\alpha}(x)}{\int_{[0,1]}\tilde h_{\beta,\alpha} \mathrm{d}\lambda} \quad \text{with} \quad
\tilde h_{\beta,\alpha}(x)=\sum_{x<\tilde T_{\beta,\alpha}^n(1)}\frac{1}{\beta^n}-\sum_{x<\tilde T_{\beta,\alpha}^n(0)}\frac{1}{\beta^n}.
\end{equation}
Furthermore:
\begin{enumerate}[{\rm(i)}]
  \item $\tilde h_{\beta,\alpha}(x) \geq 0$ for $\lambda$-almost every $x\in[0,1]$.
  \item If $\beta > \sqrt{2}$ and $\alpha\in(0,1]$, then $\tilde T_{\beta,\alpha}$ is ergodic with respect to $\mu_{\beta,\alpha}$.
\end{enumerate}
\end{theorem}

Since $\tilde T_{\beta,\alpha}$ coincides with $T_{\beta,\alpha}$ except at finitely many points, comparing \eqref{eq:density-normal-1} with \eqref{eq:density-normal} yields $\int_{[0,1]}\tilde h_{\beta,\alpha} \mathrm{d}\lambda=\int_{[0,1]}\tilde g_{\beta,\alpha} \mathrm{d}\lambda$ and $h_{\beta,\alpha}=g_{\beta,\alpha}$ for $\lambda$-a.e. $x\in[0,1]$. Then, by \eqref{eq:def-M} we have
\[
M_\beta(\alpha)=\int_{[0,1]}x h_{\beta,\alpha}(x) \mathrm{d}\lambda(x).
\]
To establish the left continuity of $M_\beta$, it suffices to prove the left continuity of the mapping $\alpha \mapsto h_{\beta,\alpha}$ in $L^1(\lambda)$. 
Note that $T_{\beta,\alpha+n}=T_{\beta,\alpha}$ and $\tilde T_{\beta,\alpha+n}=\tilde T_{\beta,\alpha}$ for all $n\in\mathbb{Z}$, which implies $M_\beta(\alpha+n)=M_\beta(\alpha)$ for $\alpha\in[0,1)$ and $n\in\mathbb{Z}$. We now prove the left continuity of $M_\beta$.

\begin{lemma}\label{lem:left-continuity-M}
Let $\beta>1$ and $\alpha\in(0,1]$. For any sequence $(\alpha_k)_{k\geq 1} \subset (0,\alpha)$ satisfying $\lim_{k\to\infty}\alpha_k=\alpha$, we have $\lim_{k\to\infty}\|h_{\beta,\alpha_k}-h_{\beta,\alpha}\|_1=0$.
\end{lemma}
The proof proceeds analogously to Lemma \ref{lem:right-continuity-M} with $\tilde g_{\beta,\alpha}(x)$ replaced by $\tilde h_{\beta,\alpha}(x)$. Lemmas \ref{lem:right-continuity-M} and \ref{lem:left-continuity-M} establish the continuity of $M_\beta$.% through the fact that $\tilde{T}_{\beta,\alpha}$ coincides with $T_{\beta,\alpha}$ except at finitely many points.  
 We now prove the central symmetry property of $M_\beta$.
\begin{lemma}\label{le:isomorphic-alpha-alpha}
Let $\psi(x)=1-x$ for $x\in[0,1]$. Given $\beta>1$ and $\alpha\in[0,1]$, 
\[
\psi\circ T_{\beta,\alpha}=T_{\beta,1-\fr{\beta-\alpha}}\circ \psi
\]
holds for Lebesgue almost every $x\in[0,1]$. Consequently, 
\[
M_\beta(\alpha) + M_\beta(1 - \fr{\beta - \alpha}) = 1.
\]
\end{lemma}

\begin{proof}
Since the case $\alpha\in [1-\fr{\beta},1)$ follows similarly, we only consider $\alpha\in [0,1-\fr{\beta})$. For $x\in[0,1]$, we have 
$T_{\beta,\alpha}(x)=\beta x+\alpha-\lfloor\beta x+\alpha\rfloor$, 
which implies
\[
\psi\circ T_{\beta,\alpha}(x)=1-T_{\beta,\alpha}(x)=1-\beta x-\alpha+\lfloor\beta x+\alpha\rfloor.
\]
On the other hand,
\begin{align*}
T_{\beta,1-\fr{\beta}-\alpha}\circ \psi(x) &= T_{\beta,1-\fr{\beta}-\alpha}(1-x) \\
&= \beta(1-x)+1-\fr{\beta}-\alpha-\lfloor\beta(1-x)+1-\fr{\beta}-\alpha\rfloor \\
&= \lfloor\beta\rfloor+1-\beta x-\alpha-\lfloor\lfloor\beta\rfloor+1-\beta x-\alpha\rfloor.
\end{align*}
For $x\in[0,1]$ with $\beta x+\alpha\notin \mathbb{N}$,
\[
\lfloor\lfloor\beta\rfloor+1-\beta x-\alpha\rfloor=\lfloor\beta\rfloor+\lfloor1-\beta x-\alpha\rfloor=\lfloor\beta\rfloor-\lfloor\beta x+\alpha\rfloor,
\]
yielding
\[
T_{\beta,1-\fr{\beta}-\alpha}\circ \psi(x)=1-\beta x-\alpha+\lfloor\beta x+\alpha\rfloor.
\]
At points where $\beta x+\alpha\in\mathbb{N}$, $T_{\beta,1-\fr{\beta}-\alpha}\circ\psi(x)=0$. Thus $\psi\circ T_{\beta,\alpha}(x)=T_{\beta,1-\fr{\beta}-\alpha}\circ\psi(x)$ except at finitely many points $\left\{\frac{i-\alpha}{\beta}: i=1,2,\ldots,\lfloor\beta\rfloor\right\}$. 

Iteratively, $\psi\circ T_{\beta,\alpha}^n=T_{\beta,\alpha}^n\circ \psi$ except on a countable set. By \eqref{eq:density-normal}, for Lebesgue almost every $x\in[0,1]$,
\begin{align*}
\tilde g_{\beta,1-\fr{\beta}-\alpha}(x) 
&= \sum_{x < T_{\beta,1-\fr{\beta}-\alpha}^n(1)}\frac{1}{\beta^n} - \sum_{x < T_{\beta,1-\fr{\beta}-\alpha}^n(0)}\frac{1}{\beta^n} \\
&= \sum_{1-x > 1 - T_{\beta,1-\fr{\beta}-\alpha}^n(1)}\frac{1}{\beta^n} - \sum_{1-x > 1 - T_{\beta,1-\fr{\beta}-\alpha}^n(0)}\frac{1}{\beta^n} \\
&= \sum_{1-x > T_{\beta,\alpha}^n(0)}\frac{1}{\beta^n} - \sum_{1-x > T_{\beta,\alpha}^n(1)}\frac{1}{\beta^n} \\
&= \sum_{1-x \leq T_{\beta,\alpha}^n(1)}\frac{1}{\beta^n} - \sum_{1-x \leq T_{\beta,\alpha}^n(0)}\frac{1}{\beta^n},
\end{align*} 
where the third equality uses the identity
\[1-T_{\beta,1-\fr{\beta}-\al}^n(z)=\psi\circ T_{\beta,1-\fr{\beta}-\al}^n(z)=T^n_{\beta,\al}\circ \psi(z)=T^n_{\beta,\al}(1-z).\]
Thus $\tilde g_{\beta,1-\fr{\beta}-\alpha}(x)=\tilde g_{\beta,\alpha}(1-x)$ in $L^1(\lambda)$, implying $g_{\beta,1-\fr{\beta}-\alpha}(x)=g_{\beta,\alpha}(1-x)$ for Lebesgue a.e. $x\in[0,1]$. By \eqref{eq:def-M}, we have 
\begin{align*}
 M_{\beta}(1-\fr{\beta}-\al)&= \int_{[0,1]}xg_{\beta,1-\fr{\beta}-\al}(x)dx=\int_{[0,1]}xg_{\beta,\alpha}(1-x)dx\\
&=\int_{[0,1]}(1-y)g_{\beta,\alpha}(y)dy=1-M_{\beta}(\alpha).\qedhere
\end{align*}
\end{proof}

By Lemma \ref{le:isomorphic-alpha-alpha}, the function $M_\beta$ is centrally symmetric about $\left(\frac{1-\fr{\beta}}{2},\frac{1}{2}\right)$ on $[0,1-\fr{\beta}]$ and about $\left(1-\frac{\fr{\beta}}{2},\frac{1}{2}\right)$ on $(1-\fr{\beta},1)$.
 Consequently, the range $\mathcal{M}_\beta$ is symmetric with respect to the line $y = 1/2$. See Figure \ref{f:no-mono} for an example.

\begin{proof}[Proof of Theorem \ref{thm:main-1-continuous}]
The result follows directly from Lemmas \ref{lem:right-continuity-M}, \ref{lem:left-continuity-M}, and \ref{le:isomorphic-alpha-alpha}.
\end{proof}

We conclude this section by noting that the continuity of $M_\beta$ can alternatively be established using the Perron-Frobenius operator $P_{\beta,\alpha} = P_{T_{\beta,\alpha}}: L^1([0,1],\lambda) \to L^1([0,1],\lambda)$, defined as (cf.~\cite{book-Boyarsky-Gora-1997})
\[
(P_{\beta,\alpha} f)(x) = \frac{\mathrm{d}}{\mathrm{d}x} \int_{T_{\beta,\alpha}^{-1}([0,x])} f  \mathrm{d}\lambda, \quad f \in L^1([0,1],\lambda).
\]

\vskip.10cm
\section{Linear property of $M_\beta(\cdot)$ on $[0,1-\fr{\beta})$.}\label{sec:matching-1}
In Theorem \ref{thm:main-1-continuous}, we established the continuity of $M_\beta$. In this section, we prove the linearity property of $M_\beta$ on $[0,1-\fr{\beta})$ for a specific class of $\beta$. Combining \eqref{eq:density-normal} and \eqref{eq:def-M}, we have 
\begin{equation}\label{eq:f-M}
\begin{split}
M_\beta(\alpha) = \int_{[0,1]}x g_{\beta,\alpha}(x)  \mathrm{d}\lambda(x) 
= \frac{1}{K_{\beta,\alpha}} \sum_{k=0}^{\infty} \frac{T_{\beta,\alpha}^k(1) - T_{\beta,\alpha}^k(0)}{\beta^k} \cdot \frac{T_{\beta,\alpha}^k(1) + T_{\beta,\alpha}^k(0)}{2},
\end{split}
\end{equation}
where $K_{\beta,\alpha} = \sum_{k=0}^{\infty} \frac{T_{\beta,\alpha}^k(1) - T_{\beta,\alpha}^k(0)}{\beta^k}$. Thus, $M_\beta(\alpha)$ is the weighted average of the sequence $\left\{ \frac{T_{\beta,\alpha}^k(1) + T_{\beta,\alpha}^k(0)}{2} \right\}_{k \geq 0}$.

If the orbits of $0$ and $1$ coincide after finitely many iterations, i.e., $T_{\beta,\alpha}^k(0) = T_{\beta,\alpha}^k(1)$ for some $k \in \mathbb{N}$, then $M_\beta(\alpha)$ reduces to a finite sum. In this case, we say $T_{\beta,\alpha}$ has \emph{matching}. If $k$ is the smallest such integer, we call $k$ the \emph{matching time}.

Given $\beta > 1$ and $\alpha \in [0,1)$, let $(a_i)$ and $(b_i)$ be the intermediate $\beta$-expansions of $0$ and $1$ under $T_{\beta,\alpha}$, respectively. Then for any $n \geq 1$, we have 
\begin{equation}\label{eq:t-0-t-1-matching}
\begin{split}
T_{\beta,\alpha}^n(0)= \beta^n \sum_{i=1}^{n} \frac{\alpha - a_i}{\beta^i}\quad \text{and}\quad 
T_{\beta,\alpha}^n(1)= \beta^n \left( 1 + \sum_{i=1}^{n} \frac{\alpha - b_i}{\beta^i} \right).
\end{split}
\end{equation}
If matching time is  $k \in \mathbb{N}$, then
\[
1 = \sum_{i=1}^{k} \frac{b_i - a_i}{\beta^i}.
\]
Thus, $T_{\beta,\alpha}$ has matching only if $\beta$ is an algebraic integer. Given $q, m \in \mathbb{N}$, let $\beta_{q,m}$ be the unique root in $(q, q+1)$ of the equation $1 = \sum_{i=1}^{m} \frac{q}{x^i}$, i.e., $\beta_{q,m}$ satisfies
\[
\beta_{q,m}^m - q \beta_{q,m}^{m-1} - q \beta_{q,m}^{m-2} - \cdots - q \beta_{q,m} - q = 0.
\]

\begin{theorem}\label{thm:linear}
For all  $q, m \in \mathbb{N}$, $M_{\beta_{q,m}}$ is linear and increasing on $[0,1-\fr{\beta_{q,m}})$.
\end{theorem}
If $m=1$, then $\beta_{q,m} = q \in \mathbb{N}$ for any $q \in \mathbb{N}$, implying $M_\beta(\alpha) = 1/2$ for all $\alpha \in [0,1)$. Hence, we assume $m \geq 2$. For $m=2$, \cite[Proposition 5.1]{Bruin-Carminati-Kalle-2017} established that $T_{\beta_{q,2}}$ exhibits matching throughout $[0,1-\fr{\beta_{q,2}})$ for every $q \in \mathbb{N}$. We now extend this result to $m \geq 3$.

\begin{lemma}\label{le:matching-m-th}
For all $q, m \in \mathbb{N}_{\geq 2}$, the transformation $T_{\beta_{q,m},\alpha}$ has matching at time $m$ for every $\alpha \in [0,1-\fr{\beta_{q,m}})$.
\end{lemma}

\begin{proof}
 For all $\beta>1$, fix $\alpha \in [0,1-\fr\beta)$. In view of  \eqref{eq:def-T-leq-1}, we have 
\[
T_{\beta,\alpha}(x) = \beta x + \alpha - i \quad \text{for} \quad x \in \Delta(i),
\]
where the partition intervals are
\[
\Delta(0) := \left[0, \frac{1-\alpha}{\beta}\right), \quad
\Delta(q) := \left[\frac{q-\alpha}{\beta}, 1\right],
\]
and for $1 \leq i \leq q-1$,
\[
\Delta(i) := \left[\frac{i-\alpha}{\beta}, \frac{i+1-\alpha}{\beta}\right).
\]

Fix $q\in \mathbb{N}$, so $\lfloor\beta_{q,m}\rfloor = q$ for all $m\geq 1$.
We prove by induction that 
\begin{equation}\label{eq:claim2}
 \begin{split}
&T_{{\beta_{q,m}},\alpha}^k(0)\in\Delta(0),\quad  T_{{\beta_{q,m}},\alpha}^k(1)\in\Delta( q )\quad \text{and}\quad T_{{\beta_{q,m}},\alpha}^k(1)-T_{{\beta_{q,m}},\alpha}^k(0)=\sum_{i=1}^{m-k}\frac{ q }{{\beta_{q,m}}^i} 
\end{split}
\end{equation}
for all $1 \leq k \leq m-1$.
For $k=1$, by \eqref{eq:t-0-t-1-matching} and $1 = \sum_{i=1}^{m} \frac{q}{{\beta_{q,m}}^i}$, we have 
\[
T_{{\beta_{q,m}},\alpha}(1) - T_{{\beta_{q,m}},\alpha}(0) = ({\beta_{q,m}} + \alpha - q) - \alpha = {\beta_{q,m}} - q = \sum_{i=1}^{m-1} \frac{q}{{\beta_{q,m}}^i} \geq \frac{q}{{\beta_{q,m}}}.
\]
%where the equality follows from ${\beta_{q,m}} = {\beta_{q,m}}_{q,m}$ satisfying $1 = \sum_{i=1}^{m} \frac{q}{{\beta_{q,m}}^i}$. 
The partition structure implies $T_{{\beta_{q,m}},\alpha}(0) \in \Delta(0)$ and $T_{{\beta_{q,m}},\alpha}(1) \in \Delta(q)$, establishing the base case. This means  \eqref{eq:claim2} holds for $m=2$.

For $m\geq 3$, assume \eqref{eq:claim2} holds for some $k \in \{1,\ldots,m-2\}$. Then
\begin{align*}
T_{{\beta_{q,m}},\alpha}^{k+1}(1) - T_{{\beta_{q,m}},\alpha}^{k+1}(0) 
&= \left( {\beta_{q,m}} T_{{\beta_{q,m}},\alpha}^k(1) + \alpha - q \right) - \left( {\beta_{q,m}} T_{{\beta_{q,m}},\alpha}^k(0) + \alpha \right) \\
&= {\beta_{q,m}} \left( T_{{\beta_{q,m}},\alpha}^k(1) - T_{{\beta_{q,m}},\alpha}^k(0) \right) - q \\
&= {\beta_{q,m}} \sum_{i=1}^{m-k} \frac{q}{{\beta_{q,m}}^i} - q = \sum_{i=1}^{m-k-1} \frac{q}{{\beta_{q,m}}^i} \geq \frac{q}{{\beta_{q,m}}}.
\end{align*}
Thus $T_{{\beta_{q,m}},\alpha}^{k+1}(0) \in \Delta(0)$ and $T_{{\beta_{q,m}},\alpha}^{k+1}(1) \in \Delta(q)$, completing the induction.

Finally, we verify matching at time $m$, i.e.,
\begin{align*}
T_{{\beta_{q,m}},\alpha}^{m}(1) -T_{{\beta_{q,m}},\alpha}^m(0)& ={\beta_{q,m}}\left(T_{{\beta_{q,m}},\alpha}^{m-1}(1) -T_{{\beta_{q,m}},\alpha}^{m-1}(0)\right) -q=0.     \qedhere
%= {\beta_{q,m}}^m \left( 1 - \sum_{i=1}^{m} \frac{q - \alpha}{{\beta_{q,m}}^i} \right) 
%= {\beta_{q,m}}^m \sum_{i=1}^{m} \frac{\alpha}{{\beta_{q,m}}^i} 
%= T_{{\beta_{q,m}},\alpha}^m(0). \qedhere
\end{align*}
\end{proof}

Motivated by the preceding proof, we introduce the following definition of  matching intervals.
\begin{definition}\label{def-matching}
Given $\beta>1$ and two blocks $a_1\ldots a_{k+1}, b_1\ldots b_{k+1}\in\set{0,1,\ldots, \lf\beta\rf+1}^{k+1}$, define the \emph{matching interval} associated with $(a_1\ldots a_{k+1}, b_1\ldots b_{k+1})$ as
\begin{align*}
I_\beta(a_1\ldots a_{k+1}, b_1\ldots b_{k+1}) := \Bigg\{ \alpha \in [0,1) : 
& T_{\beta,\alpha} \text{ has matching at time } k+1 \text{ and } \\
& T_{\beta,\alpha}^{i-1}(0) \in \Delta(a_i), ~ T_{\beta,\alpha}^{i-1}(1) \in \Delta(b_i) \quad \forall~ 1 \le i \le k+1 \Bigg\}.
\end{align*}
\end{definition}
\begin{lemma}\label{le:linea-on-matching-interval}
Let $\beta>1$. For any blocks $a_1\ldots a_{k+1}, b_1\ldots b_{k+1}\in\set{0,1,\ldots, \lf\beta \rf+1}^{k+1}$, the function $M_{\beta}$ is linear on $I_{\beta}(a_1\ldots a_{k+1}, b_1\ldots b_{k+1})$.
\end{lemma}
\begin{proof}
For all  $\alpha \in I_{\beta}(a_1\ldots a_{k+1}, b_1\ldots b_{k+1})$, we have for  $1 \le i \le k$,
\[
T_{\beta,\alpha}^i(0) = \beta^i \left( 0 - \sum_{j=1}^{i} \frac{a_j - \alpha}{\beta^j} \right), \quad 
T_{\beta,\alpha}^i(1) = \beta^i \left( 1 - \sum_{j=1}^{i} \frac{b_j - \alpha}{\beta^j} \right),
\]
and $T_{\beta,\alpha}^i(0) = T_{\beta,\alpha}^i(1)$ for all $i \ge k+1$. Adopt the convention $\sum_{j=1}^0 c_j = 0$ for any sequence $(c_j)$. Then for $0 \le i \le k$,
\begin{equation}\label{eq:la-i}
\frac{T_{\beta,\alpha}^i(1) - T_{\beta,\alpha}^i(0)}{\beta^i} = 1 + \sum_{j=1}^{i} \frac{a_j - b_j}{\beta^j} =: \lambda_i
\end{equation}
is independent of $\alpha$, and 
\begin{equation}\label{eq:g-i}
\frac{T_{\beta,\alpha}^i(1) + T_{\beta,\alpha}^i(0)}{2} = \frac{ \beta^i \left( 1 - \sum_{j=1}^{i} \frac{a_j + b_j}{\beta^j} + \sum_{j=1}^{i} \frac{2\alpha}{\beta^j} \right) }{2} =: g_i(\alpha)
\end{equation}
is affine in $\alpha$. Therefore,
\begin{align*}
M_\beta(\alpha) &= \frac{1}{\sum_{i=0}^{\infty} \frac{T_{\beta,\alpha}^i(1) - T_{\beta,\alpha}^i(0)}{\beta^i}} \sum_{i=0}^{\infty} \frac{(T_{\beta,\alpha}^i(1))^2 - (T_{\beta,\alpha}^i(0))^2}{2\beta^i} \\
&= \frac{1}{\sum_{i=0}^{k} \frac{T_{\beta,\alpha}^i(1) - T_{\beta,\alpha}^i(0)}{\beta^i}} \sum_{i=0}^{k} \frac{T_{\beta,\alpha}^i(1) - T_{\beta,\alpha}^i(0)}{\beta^i} \cdot \frac{T_{\beta,\alpha}^i(1) + T_{\beta,\alpha}^i(0)}{2} \\
&= \frac{1}{\sum_{i=0}^{k} \lambda_i} \sum_{i=0}^{k} \lambda_i g_i(\alpha),
\end{align*}
which is linear in $\alpha$ on $I_{\beta}(a_1\ldots a_{k+1}, b_1\ldots b_{k+1})$.
\end{proof}

\begin{proof}[Proof of Theorem \ref{thm:linear}]
For all integers  $q \geq 1$ and $m \geq 2$, from the proof of Lemma \ref{le:matching-m-th}, we have $[0,1-\fr {\beta_{q,m}}) \subset I_{{\beta_{q,m}}}(0^m, q^m)$. By Lemma \ref{le:linea-on-matching-interval}, $M_{{\beta_{q,m}}}$ is linear on $[0,1-\fr {\beta_{q,m}})$. 

Observe that $\lambda_0 = 1$, and for $1 \leq i \leq m-1$,
\[
\lambda_i = 1 + \sum_{j=1}^{i} \frac{0 - q}{{\beta_{q,m}}^j} = \sum_{j=1}^{m} \frac{q}{{\beta_{q,m}}^j} - \sum_{j=1}^{i} \frac{q}{{\beta_{q,m}}^j} > 0,
\]
where the second equality holds because ${\beta_{q,m}} = {\beta_{q,m}}_{q,m}$ satisfies $1 = \sum_{j=1}^{m} \frac{q}{{\beta_{q,m}}^j}$. Consequently, the coefficient of $\alpha$ in $M_{\beta_{q,m}}(\alpha)$ is positive, implying that $M_{{\beta_{q,m}}}$ is strictly increasing and linear on $[0,1-\fr {\beta_{q,m}})$.
\end{proof}
Given $q, m \in \mathbb{N}$, recall that $\beta_{q,m}$ is the unique root in $(q, q+1)$ of $1 = \sum_{i=1}^{m} \frac{q}{x^i}$, satisfying
\[
\beta_{q,m}^m - q\beta_{q,m}^{m-1} - q\beta_{q,m}^{m-2} - \cdots - q\beta_{q,m} - q = 0.
\]
Fixed $q \in \mathbb{N}$, the sequence $(\beta_{q,m})_{m \geq 2}$ increases to $q+1$. By Theorem \ref{thm:linear},
\[
M_{\beta_{q,m}}(\alpha) = k_{q,m}\alpha + b_{q,m} \quad \text{for} \quad \alpha \in [0, 1 - \fr{\beta_{q,m}}).
\]
At the end of this section, we establish the strict monotonicity of $(k_{q,m})_m$ and $(b_{q,m})_m$. 
First, we derive explicit formulas for $k_{q,m}$ and $b_{q,m}$.

For $\alpha \in [0, 1 - \fr{\beta_{q,m}})=[0, q+1 - \beta_{q,m})$, the proof of Lemma \ref{le:linea-on-matching-interval} gives
\begin{equation}\label{eq:M-ell-m-f}
M_{\beta_{q,m}}(\alpha) = \frac{1}{\sum_{i=0}^{m-1} \lambda_i} \sum_{i=0}^{m-1} \lambda_i g_i(\alpha),
\end{equation}
where $\lambda_i$ is constant for $\alpha \in [0, q+1 - \beta_{q,m})$ and $0 \leq i \leq m-1$. Denote by
\[
K_{q,m} := K_{\beta_{q,m},\alpha} = \frac{1}{\sum_{i=0}^{m-1} \lambda_i}.
\]
Since $[0, 1 - \fr{\beta_{q,m}}) \subset I_{\beta_{q,m}}(0^m, q^m)$, equation \eqref{eq:la-i} yields
\[
K_{q,m} = \frac{1}{m - \frac{(m-1)q}{\beta_{q,m}} - \frac{(m-2)q}{\beta_{q,m}^2} - \cdots - \frac{q}{\beta_{q,m}^{m-1}}}.
\]
Set $S := \frac{(m-1)q}{\beta_{q,m}} + \frac{(m-2)q}{\beta_{q,m}^2} + \cdots + \frac{q}{\beta_{q,m}^{m-1}}$. Then
\[
(\beta_{q,m} - 1)S = q(m-1) - \frac{q}{\beta_{q,m}} - \cdots - \frac{q}{\beta_{q,m}^{m-1}} = q(m-1) - \frac{q(\beta_{q,m}^{m-1} - 1)}{\beta_{q,m}^{m-1}(\beta_{q,m} - 1)},
\]
implying
\[
S = \frac{q(m-1)}{\beta_{q,m} - 1} - \frac{q(\beta_{q,m}^{m-1} - 1)}{(\beta_{q,m} - 1)^2 \beta_{q,m}^{m-1}}.
\]
Combined with
\[
\beta_{q,m} - q = \sum_{k=1}^{m-1} \frac{q}{\beta_{q,m}^k} = \frac{q(\beta_{q,m}^{m-1} - 1)}{(\beta_{q,m} - 1)\beta_{q,m}^{m-1}},
\]
we obtain
\[
K_{q,m} = \frac{1}{m - \frac{(m-1)q}{\beta_{q,m} - 1} + \frac{\beta_{q,m} - q}{\beta_{q,m} - 1}}.
\]
Using $\beta_{q,m} - 1 = q \left(1 - \frac{1}{\beta_{q,m}^m}\right)$, we simplify
\[
m - \frac{(m-1)q}{\beta_{q,m} - 1} + \frac{\beta_{q,m} - q}{\beta_{q,m} - 1} = \frac{\beta_{q,m} - \frac{mq}{\beta_{q,m}^m}}{\beta_{q,m} - 1},
\]
yielding the closed-form expression
\begin{equation}\label{eq:K-n-m}
K_{q,m} = \frac{\beta_{q,m} - 1}{\beta_{q,m} - \frac{mq}{\beta_{q,m}^m}}.
\end{equation}

Since $[0, 1 - \fr{\beta_{q,m}}) \subset I_{\beta_{q,m}}(0^m, q^m)$, equations \eqref{eq:la-i} and \eqref{eq:g-i} yield
\begin{equation}\label{eq:M-ell-m}
\begin{split}
\sum_{i=0}^{m-1} \lambda_i g_i(\alpha) 
= &\sum_{i=0}^{m-1} \frac{\beta_{q,m}^i \left(1 - \sum_{j=1}^{i} \frac{q}{\beta_{q,m}^j}\right) \left(1 - \sum_{j=1}^{i} \frac{q}{\beta_{q,m}^j} + \sum_{j=1}^{i} \frac{2}{\beta_{q,m}^j} \alpha \right)}{2} \\
= &\sum_{i=1}^{m-1} \left( \sum_{j=1}^i \frac{1}{\beta_{q,m}^j} \right) \left( \sum_{j=1}^{m-i} \frac{q}{\beta_{q,m}^j} \right) \alpha + \frac{1}{2} + \sum_{i=1}^{m-1} \frac{\beta_{q,m}^i \left( \sum_{j=i+1}^{m} \frac{q}{\beta_{q,m}^j} \right)^2}{2},
\end{split}
\end{equation}
where $\sum_{j=1}^{0} \frac{q}{\beta_{q,m}^j} = 0$ by convention, and the term $1/2$ arises from the case $i=0$.

Now we give the formula of $k_{ q ,m}$. From \eqref{eq:M-ell-m-f} and\eqref{eq:M-ell-m}, we expand $k_{ q ,m}$ as
\begin{equation*}
\begin{split}
k_{ q ,m} 
&= K_{ q ,m} \sum_{i=1}^{m-1} \left(\frac{1}{\beta_{ q ,m}} + \cdots + \frac{1}{\beta_{ q ,m}^i}\right)\!\left(\frac{ q }{\beta_{ q ,m}} + \cdots + \frac{ q }{\beta_{ q ,m}^{m-i}}\right) \\
&= \frac{ q  K_{ q ,m}}{(\beta_{ q ,m} - 1)^2} \sum_{i=1}^{m-1} \left(1 - \frac{1}{\beta_{ q ,m}^i} - \frac{1}{\beta_{ q ,m}^{m-i}} + \frac{1}{\beta_{ q ,m}^m}\right) \\
&%= \frac{ q  K_{ q ,m} \left((m-1) + \frac{m-1}{\beta_{ q ,m}^m} - \sum_{i=1}^{m-1}\frac{2}{\beta_{ q ,m}^i}\right)}{(\beta_{ q ,m} - 1)^2}
=\frac{ q  K_{ q ,m} \left((m-1) + \frac{m+1}{\beta_{ q ,m}^m} - \frac{2}{ q }\right)}{(\beta_{ q ,m} - 1)^2},
\end{split}
\end{equation*}
where the last equal follows from $1=\sum_{i=1}^m\frac{ q }{\beta_{ q ,m}^i}$.
%Substituting $\beta_{ q ,m} - 1 = n\left(1 - \frac{1}{\beta_{ q ,m}^m}\right)$, we reduce:
%\begin{equation*}
%(m-1) + \frac{m-1}{\beta_{ q ,m}^m} - \frac{2\left(1 - \frac{1}{\beta_{ q ,m}^{m-1}}\right)}{\beta_{ q ,m}-1} 
%= (m-1) + \frac{m+1}{\beta_{ q ,m}^m} - \frac{2}{ q }.
%\end{equation*}
Combining with~\eqref{eq:K-n-m}, we obtain
\begin{equation}\label{eq:k-n-m}
k_{ q ,m} = \frac{ q  \left((m-1) + \frac{m+1}{\beta_{ q ,m}^m} - \frac{2}{ q }\right)}{(\beta_{ q ,m} - 1)\left(\beta_{ q ,m} - \frac{m q }{\beta_{ q ,m}^m}\right)}.
\end{equation}

Analogous derivations for $m+1$ yield that 
\begin{equation}\label{eq:K-k-n-m+1}
\begin{split}
K_{ q ,m+1} 
&= \frac{1}{m+1 - \frac{m q }{\beta_{ q ,m+1}-1} + \frac{\beta_{ q ,m+1}- q }{\beta_{ q ,m+1}-1}} 
= \frac{\beta_{ q ,m+1} - 1}{\beta_{ q ,m+1} - \frac{(m+1) q }{\beta_{ q ,m+1}^{m+1}}}, \\
k_{ q ,m+1} 
&= \frac{ q \left(m + \frac{m+2}{\beta_{ q ,m+1}^{m+1}} - \frac{2}{ q }\right)}{(\beta_{ q ,m+1}-1)\left(\beta_{ q ,m+1} - \frac{(m+1) q }{\beta_{ q ,m+1}^{m+1}}\right)}.
\end{split}
\end{equation}

We now systematically derive the closed-form expression for $d_{ q ,m}$. By  the foundational relations in \eqref{eq:M-ell-m-f} and \eqref{eq:M-ell-m}, we develop the expansion through successive transformations,
\begin{equation*}
\begin{split}
d_{ q ,m}=&%\frac{K_{ q ,m}}{2} \sum_{i=0}^{m-1} \beta_{ q ,m}^i\left(\sum_{j=i+1}^{m}\frac{ q }{\beta_{ q ,m}^j }\right)^2=
\frac{K_{ q ,m}}{2} \left(1+\sum_{i=1}^{m-1} \beta_{ q ,m}^i\left(\sum_{j=i+1}^{m}\frac{ q }{\beta_{ q ,m}^j }\right)^2\right)\\
=&\frac{K_{ q ,m}}{2} \left(1+\frac{ q ^2}{({\beta_{ q ,m}}-1)^2}\sum_{i=1}^{m-1}\frac{1-\frac{2}{{\beta_{ q ,m}}^i}+\frac{1}{{\beta_{ q ,m}}^{2i}}}{{\beta_{ q ,m}}^{m-i}}\right)\\
=&\frac{K_{ q ,m}}{2} \left(1+\frac{ q ^2}{({\beta_{ q ,m}}-1)^2}\left(\sum_{i=1}^{2m-1}\frac{1}{{\beta_{ q ,m}}^i}
-\frac{2m-1}{{\beta_{ q ,m}}^m}\right)\right)\\
%=&\frac{K_{ q ,m}}{2} \left(1+\frac{ q ^2}{({\beta_{ q ,m}}-1)^2}\left(\sum_{i=1}^{2m-1}\frac{1}{{\beta_{ q ,m}}^{i}}
%-\frac{2(m-1)+1}{{\beta_{ q ,m}}^m}\right)\right)\\
=&\frac{K_{ q ,m}}{2} \left(1+\frac{ q ^2}{({\beta_{ q ,m}}-1)^2}\left(\frac{1}{{\beta_{ q ,m}}-1}
-\frac{1}{{\beta_{ q ,m}}^{2m-1}({\beta_{ q ,m}}-1)}-\frac{2m-1}{{\beta_{ q ,m}}^m}\right)\right)\\
=&\frac{(\beta_{ q ,m}-1)+\frac{ q ^2}{({\beta_{ q ,m}}-1)^2}
-\frac{ q ^2}{{\beta_{ q ,m}}^{2m-1}({\beta_{ q ,m}}-1)^2}-\frac{ q ^2(2m-1)}{{\beta_{ q ,m}}^m}}
{2(\beta_{ q ,m}-\frac{m q }{\beta_{ q ,m}^m})},\\
%=&\frac{K_{ q ,m}}{2} \left(1+\frac{ q ^2}{({\beta_{ q ,m}}-1)^3}
%-\frac{ q ^2}{{\beta_{ q ,m}}^{2m-1}({\beta_{ q ,m}}-1)^3}-\frac{ q ^2(2m-1)}{{\beta_{ q ,m}}^m({\beta_{ q ,m}}-1)}\right)\\
\end{split}
\end{equation*}
where the final simplification employs the definition of $K_{ q ,m}$ from \eqref{eq:K-n-m}.

To further reduce the expression, we simplify the following term
\begin{equation*}\label{eq:key-simplification}
\frac{ q ^2}{(\beta_{ q ,m}-1)^2}\left(1 - \frac{1}{\beta_{ q ,m}^{2m-1}}\right) = \frac{\beta_{ q ,m}^{2m} - 1}{(\beta_{ q ,m}^{m}-1)^2\beta_{ q ,m}^{2m-1}} = 1 + \frac{2}{\beta_{ q ,m}^{m}-1} - \frac{ q }{\beta_{ q ,m}^{m}(\beta_{ q ,m}^{m}-1)}.
\end{equation*}
Finally, we employ the characteristic equation $\beta_{ q ,m}-1 =  q  - \frac{ q }{\beta_{ q ,m}^{m}}$ to  achieve the critical term simplification,
\begin{align}\label{eq:d-n-m}
d_{ q ,m} &= \frac{ q +1 + \frac{2}{\beta_{ q ,m}^{m}-1} - \frac{ q }{\beta_{ q ,m}^{m}(\beta_{ q ,m}^{m}-1)} - \frac{2 q ^2(m-1) + q }{\beta_{ q ,m}^{m}}}{2\left( q +1 - \frac{ q  (m+1)}{\beta_{ q ,m}^m}\right)} \nonumber \\
&= \frac 12+\frac{ \frac{2}{\beta_{ q ,m}^{m}-1} - \frac{ q }{\beta_{ q ,m}^{m}(\beta_{ q ,m}^{m}-1)} - \frac{2 q ^2(m-1) -  q  m}{\beta_{ q ,m}^{m}}}{2\left( q +1 - \frac{ q  (m+1)}{\beta_{ q ,m}^m}\right)}
\end{align}

The recursive structure becomes apparent when considering the $(m+1)$-dimensional case:
\begin{equation}\label{eq:d-n-m+1}
d_{ q ,m+1} = \frac 12+\frac{\frac{2}{\beta_{ q ,m+1}^{m+1}-1} - \frac{ q }{\beta_{ q ,m+1}^{m+1}(\beta_{ q ,m+1}^{m+1}-1)} - \frac{2 q ^2m- q (m+1)}{\beta_{ q ,m+1}^{m+1}}}{2\left( q +1 - \frac{ q  (m+2)}{\beta_{ q ,m+1}^{m+1}}\right)}.
\end{equation}

Now we establish the relationship between $\beta_{ q ,m}$ and $\beta_{ q ,m+1}$, which is crucial for proving the inequalities $k_{ q ,m+1} > k_{ q ,m}$ and $d_{ q ,m+1} > d_{ q ,m}$.

 \begin{lemma}\label{cl:b-nm-bnm+1}
 For all $ q \geq 1$ and $m\geq 2$, the following inequality holds
 $$\frac{ q }{\beta_{ q ,m+1}^{m}}>\beta_{ q ,m+1}-\beta_{ q ,m}>\frac{ q  (1-\frac{1}{\beta_{ q ,m}})}{\beta_{ q ,m+1}^{m}}.$$
 \end{lemma}
 \begin{proof}
 The left inequality follows directly from the monotonicity $\beta_{ q ,m} < \beta_{ q ,m+1}$ and the definition of $\beta_{ q , m}$. From those, we have 
$$\beta_{ q ,m+1}-\beta_{ q ,m}= q +\frac{ q }{\beta_{ q ,m+1}}+\ldots +\frac{ q }{\beta_{ q ,m+1}^{m}}
-\left( q +\frac{ q }{\beta_{ q ,m}}+\ldots +\frac{ q }{\beta_{ q ,m}^{m-1}}\right)<\frac{ q }{\beta_{ q ,m+1}^{m}}.$$

 For the right inequality, consider the following difference,
\begin{equation}\label{eq:11}
\begin{split}
&\beta_{ q ,m+1}-\beta_{ q ,m}\\
%稍后删掉这一步
%=&\frac{ q }{\beta_{ q ,m+1}^{m}}- q \left(\frac{\beta_{ q ,m+1}-\beta_{ q ,m}}{\beta_{ q ,m+1}\beta_{ q ,m}}
%+\frac{\beta_{ q ,m+1}^2-\beta_{ q ,m}^2}{\beta_{ q ,m+1}^2\beta_{ q ,m}^2}+\ldots +
%\frac{\beta_{ q ,m+1}^{m-1}-\beta_{ q ,m}^{m-1}}{\beta_{ q ,m+1}^{m-1}\beta_{ q ,m}^{m-1}}\right)\\
%===========
=&\frac{ q }{\beta_{ q ,m+1}^{m}}- q  (\beta_{ q ,m+1}-\beta_{ q ,m})\left(\frac{1}{\beta_{ q ,m+1}\beta_{ q ,m}}
+\frac{\beta_{ q ,m+1}+\beta_{ q ,m}}{\beta_{ q ,m+1}^2\beta_{ q ,m}^2}+\ldots +\right.\\
&\left.\frac{\beta_{ q ,m+1}^{m-1}+\beta_{ q ,m+1}^{m-2}\beta_{ q ,m}+\ldots +
\beta_{ q ,m}^{m-1}}{\beta_{ q ,m+1}^{m-1}\beta_{ q ,m}^{m-1}}\right)\\
>&\frac{ q }{\beta_{ q ,m+1}^{m}}- q  (\beta_{ q ,m+1}-\beta_{ q ,m})\left(\frac{1}{\beta_{ q ,m+1}\beta_{ q ,m}}
+\frac{2\beta_{ q ,m+1}}{\beta_{ q ,m+1}^2\beta_{ q ,m}^2}+\ldots +
\frac{(m-1)\beta_{ q ,m+1}^{m-2}}{\beta_{ q ,m+1}^{m-1}\beta_{ q ,m}^{m-1}}\right)\\
=&\frac{ q }{\beta_{ q ,m+1}^{m}}-\frac{ q  (\beta_{ q ,m+1}-\beta_{ q ,m})}{\beta_{ q ,m+1}}\left(\frac{1}{\beta_{ q ,m}}
+\frac{2}{\beta_{ q ,m}^2}+\ldots +
\frac{m-1}{\beta_{ q ,m}^{m-1}}\right).\\
%=&\frac{ q }{\beta_{ q ,m+1}^{m}}-\frac{ q  (\beta_{ q ,m+1}-\beta_{ q ,m})}{\beta_{ q ,m+1}}
%\left(\frac{\beta_{ q ,m}-\frac{1}{\beta_{ q ,m}^{m-1}}}{(\beta_{ q ,m}-1)^2}-
%\frac{m}{\beta_{ q ,m}^m(\beta_{ q ,m}-1)}\right).\\
%>&\frac{ q }{\beta_{ q ,m+1}^{m}\left(1+\frac{ q  (\beta_{ q ,m}^m-1)}{(\beta_{ q ,m}-1)^2\beta_{ q ,m}^m}\right)},
\end{split}
\end{equation}
To resolve the summation, let $S_1 :=\sum_{k=1}^{m-1} \frac{k}{\beta_{ q ,m}^k}$. Through geometric series manipulation, we have 
\begin{align*}
(\beta_{ q ,m}-1)S_1=1+\frac{1}{\beta_{ q ,m}}+\frac{1}{\beta_{ q ,m}^2}+\ldots +\frac{1}{\beta_{ q ,m}^{m-2}}-\frac{(m-1)}{\beta_{ q ,m}^{m-1}}
%\\=&1+\frac{1}{\beta_{ q ,m}}+\ldots +\frac{1}{\beta_{ q ,m}^{m-2}}+\frac{1}{\beta_{ q ,m}^{m-1}}-\frac{m}{\beta_{ q ,m}^{m-1}}
=\frac{ \beta_{ q ,m}-\frac{1}{ \beta_{ q ,m}^{m-1}}}{( \beta_{ q ,m}-1)}-\frac{m}{ \beta_{ q ,m}}.
\end{align*}
This yields $S_1 = \frac{\beta_{ q ,m} - \beta_{ q ,m}^{-(m-1)}}{(\beta_{ q ,m} - 1)^2} - \frac{m}{\beta_{ q ,m}^m (\beta_{ q ,m} - 1)}$.
Substituting back into \eqref{eq:11}, we obtain
\begin{equation*}
\beta_{ q ,m+1} - \beta_{ q ,m} > \frac{ q }{\beta_{ q ,m+1}^m} \left[ 1 + \frac{ q }{\beta_{ q ,m+1}} \left( \frac{\beta_{ q ,m} - \beta_{ q ,m}^{-(m-1)}}{(\beta_{ q ,m} - 1)^2} - \frac{m}{\beta_{ q ,m}^m (\beta_{ q ,m} - 1)} \right) \right]^{-1}.
\end{equation*}
Through successive approximations, we conclude that 
\begin{equation}\label{eq:b-m+1-b-m-leq}
\begin{split}
\beta_{ q ,m+1}-\beta_{ q ,m}
%>&\frac{ q }{\beta_{ q ,m+1}^{m}\left(1+{ q }\left(\frac{1-\frac{1}{\beta_{ q ,m}^{m}}}{(\beta_{ q ,m}-1)^2}-
%\frac{m}{\beta_{ q ,m}^{m+1}(\beta_{ q ,m}-1)}\right)\right)}\\
>&\frac{ q }{\beta_{ q ,m+1}^{m}\left(1+\frac{ q  (\beta_{ q ,m}^m-1)}{(\beta_{ q ,m}-1)^2\beta_{ q ,m}^m}\right)}
=\frac{ q }{\beta_{ q ,m+1}^{m}\left(1+\frac{ q  (\beta_{ q ,m}-1)(\beta_{ q ,m}^{m-1}+\ldots +1)}{(\beta_{ q ,m}-1)^2\beta_{ q ,m}^m}\right)}\\
=&\frac{ q }{\beta_{ q ,m+1}^{m}}\frac{\beta_{ q ,m}-1}{\beta_{ q ,m}}
=\frac{ q  (1-\frac{1}{\beta_{ q ,m}})}{\beta_{ q ,m+1}^{m}},
\end{split}
\end{equation}
which completes the proof of Lemma \ref{cl:b-nm-bnm+1}.
\end{proof}

\begin{proof}[Proof of Theorem \ref{thm:main-3-increasing}]
Fix $q > 1$. To prove the coefficient sequence $(k_{q,m})$ is increasing, we need to show $k_{q,m+1} > k_{q,m}$. By \eqref{eq:k-n-m} and \eqref{eq:K-k-n-m+1}, this reduces to establishing that
\begin{equation}\label{eq:key-inequality}
\begin{aligned}
&(\beta_{q,m}-1)\left(\beta_{q,m}-\frac{mq}{\beta_{q,m}^m}\right)\left(m + \frac{m+2}{\beta_{q,m+1}^{m+1}} - \frac{2}{q}\right) \\
&\quad - (\beta_{q,m+1}-1)\left(\beta_{q,m+1}-\frac{(m+1)q}{\beta_{q,m+1}^{m+1}}\right)\left((m-1) + \frac{m+1}{\beta_{q,m}^m} - \frac{2}{q}\right) > 0.
\end{aligned}
\end{equation}

Applying Lemma \ref{cl:b-nm-bnm+1}, we transform the left-hand side as follows,
\begin{equation*}%\label{eq:transformed-expression}
\begin{aligned}
&\frac{(\beta_{q,m}-1)\left(\beta_{q,m}-\dfrac{mq}{\beta_{q,m}^m}\right)}{\beta_{q,m+1}^{m+1}} 
+ \beta_{q,m+1}(\beta_{q,m+1}-1) 
- \frac{q(m+1)(\beta_{q,m+1}-1)}{\beta_{q,m+1}^{m+1}} \\
&\quad - \left(m - \frac{2}{q}\right)(\beta_{q,m+1} - \beta_{q,m})(\beta_{q,m} + \beta_{q,m+1} - 1) - \frac{mq\left(m -\dfrac{2}{q}\right)(\beta_{q,m}-1)}{\beta_{q,m}^m} \\
&\quad
+ \frac{(m+1)q\left(m - \dfrac{2}{q}\right)(\beta_{q,m+1}-1)}{\beta_{q,m+1}^{m+1}}  + \frac{(m+1)(\beta_{q,m}-1)\left(\beta_{q,m}-\dfrac{mq}{\beta_{q,m}^m}\right)}{\beta_{q,m+1}^{m+1}} \\
&\quad - \frac{(m+1)(\beta_{q,m+1}-1)\left(\beta_{q,m+1}-\dfrac{(m+1)q}{\beta_{q,m+1}^{m+1}}\right)}{\beta_{q,m}^m} \\
>\quad  &\beta_{q,m+1}(\beta_{q,m+1}-1) 
- \frac{mq\left(m - \dfrac{2}{q}\right)(\beta_{q,m}-1) + (m+1)\beta_{q,m+1}(\beta_{q,m+1}-1)}{\beta_{q,m}^m} \\
&\quad + \frac{q(m+1)\left(m-1 - \dfrac{2}{q}\right)(\beta_{q,m+1}-1) + (m+2)(\beta_{q,m}-1)\left(\beta_{q,m}-\dfrac{mq}{\beta_{q,m}^m}\right)}{\beta_{q,m+1}^{m+1}} \\
&\quad - \frac{q\beta_{q,m+1}\left(m - \dfrac{2}{q}\right)(\beta_{q,m} + \beta_{q,m+1} - 1)}{\beta_{q,m+1}^{m+1}} \\
>\quad & q^2 - \frac{(m^2 + m + 1)q^2}{\beta_{q,m}^m}.
\end{aligned}
\end{equation*}
It's easy to check the final expression $ q ^2 - \frac{(m^2 + m + 1) q ^2}{\beta_{ q ,m}^m}$ is positive for the following cases:
\begin{itemize}
\item all $q \geq 3$;
\item $q = 2$ with $m \geq 3$;
\item $q = 1$ with $m \geq 6$.
\end{itemize}
For the remaining exceptional cases, the positivity can be verified by direct computation of specific values.

We now prove $(d_{q,m})$ is increasing, which   is equivalent to  $d_{q,m+1} - d_{q,m} > 0$.
By equations \eqref{eq:d-n-m} and \eqref{eq:d-n-m+1}, it reduces to establish that 
\begin{align*}
&\left[ q + 1 - \frac{q(m+1)}{\beta_{q,m}^m} \right]
\left[ \frac{2}{\beta_{q,m+1}^{m+1} - 1} 
- \frac{q}{\beta_{q,m+1}^{m+1} (\beta_{q,m+1}^{m+1} - 1)} 
- \frac{2q^2m - q(m+1)}{\beta_{q,m+1}^{m+1}} \right] \nonumber \\
&\quad - \left[ q + 1 - \frac{q(m+2)}{\beta_{q,m+1}^{m+1}} \right]
\left[ \frac{2}{\beta_{q,m}^m - 1} 
- \frac{q}{\beta_{q,m}^m (\beta_{q,m}^m - 1)} 
- \frac{2q^2(m-1) - qm}{\beta_{q,m}^m} \right] > 0.
\end{align*}

Set \(x := \beta_{q,m}^m\) and \(y := \beta_{q,m+1}^{m+1}\) to simplify notation. 
The above expression transforms to 
\begin{align*}
& \left[  q  + 1 - \frac{ q (m+1)}{x} \right] \left[ \frac{2}{y-1} - \frac{ q }{y(y-1)} - \frac{2 q ^2m -  q (m+1)}{y} \right]-\\
&\left[  q  + 1 - \frac{ q (m+2)}{y} \right] \left[ \frac{2}{x-1} - \frac{ q }{x(x-1)} - \frac{2 q ^2(m-1) -  q  m}{x} \right]\\
=& \frac{2( q +1)}{y-1} 
- \frac{ q ( q +1)}{y(y-1)} 
- \frac{( q +1)(2 q ^2m- q  (m+1))}{y} \\
&\quad - \left[ \frac{2( q +1)}{x-1} 
- \frac{ q ( q +1)}{x(x-1)} 
- \frac{( q +1)(2 q ^2(m-1) -  q  m)}{x} \right] \\
&+ \frac{2 q (m+2)}{y(x-1)} 
- \frac{ q ^2(m+2)}{yx(x-1)}  - \frac{ q (m+2)(2 q ^2(m-1)+ q  m)}{yx} \\
&\quad - \left[
\frac{2 q (m+1)}{x(y-1)} - \frac{ q ^2(m+1)}{xy(y-1)}- \frac{ q (m+1)(2 q ^2m+ q  (m+1))}{yx}\right] \\
%=& \frac{2( q +1)(x-y)}{(x-1)(y-1)} 
%- \frac{ q ( q +1)\left[x(x-1)-y(y-1)\right]}{xy(x-1)(y-1)}\\
% &- \frac{( q +1)(2 q ^2m- q  m)(x-y)}{xy}  
%+ \frac{ q ( q +1)}{y} 
%- \frac{2 q ^2( q +1)}{x} \\
%&\quad + \frac{2 q (m+1)(y-x)}{xy(x-1)(y-1)} - \frac{ q ^2(m+1)\left[(y-1)-(x-1)\right]}{xy(x-1)(y-1)}  \\
%&\quad + \frac{2 q }{y(x-1)} 
%- \frac{ q ^2}{yx(x-1)}+\frac{ q ^2(m+1)(2 q +1)- q (2 q ^2(m-1)+ q  m)}{xy}.\\
%=& \frac{2( q +1)(x-y)}{(x-1)(y-1)} 
%- \frac{ q ( q +1)\left[x(x-1)-y(y-1)\right]}{xy(x-1)(y-1)}\\
% &- \frac{( q +1)(2 q ^2m- q  m)(x-y)}{xy}  
%+ \frac{ q ( q +1)}{y} 
%- \frac{2 q ^2( q +1)}{x} \\
%&\quad + \frac{(2- q ) q (m+1)(y-x)}{xy(x-1)(y-1)}  \\
%&\quad + \frac{2 q }{y(x-1)} 
%- \frac{ q ^2}{yx(x-1)}+\frac{ q ^2(m+1)( q +1)- q (2 q ^2(m-1)+ q  m)}{xy}\\
%=& \frac{( q +1)(x-y)(2xy- q (x+y-1)}{xy(x-1)(y-1)} \\
% &- \frac{ q  m( q +1)(2 q -1)(x-y)}{xy}  
%+ \frac{ q ( q +1)}{y} 
%- \frac{2 q ^2( q +1)}{x} \\
%&\quad + \frac{(2- q ) q (m+1)(y-x)}{xy(x-1)(y-1)}  \\
%&\quad + \frac{2 q }{y(x-1)} 
%- \frac{ q ^2}{yx(x-1)}+\frac{ q ^2(3 q -m q +1)}{xy}.
=& \frac{( q +1)(x-y)(2xy- q (x+y-1)}{xy(x-1)(y-1)}
- \frac{ q  m( q +1)(2 q -1)(x-y)}{xy} \\ 
&+ \frac{ q ( q +1)}{y} 
- \frac{2 q ^2( q +1)}{x} 
 + \frac{(2- q ) q (m+1)(y-x)}{xy(x-1)(y-1)}  \\
&\quad + \frac{2 q }{y(x-1)} 
- \frac{ q ^2}{yx(x-1)}+\frac{ q ^2(3 q -m q +1)}{xy}.
\end{align*}

For $x$ and $y$, we have the inequality chain $
y - 1 > y - x > q \beta_{q,m+1}^m$. 
We analyze each component of the above target expression by using this chain,
\begin{align*}
 0>\frac{( q +1)(x-y)(2xy- q (x+y-1)}{xy(x-1)(y-1)}& >-\frac{( q +1)(2xy- q (x+y-1)}{xy(x-1)}>-\frac{2( q +1)}{x},\\
%where the first inequality hold since $y-x> q  \beta_{ q ,m+1}^m$.
- \frac{ q  m( q +1)(2 q -1)(x-y)}{xy} &> \frac{ q  m( q +1)(2 q -1) q  \beta_{ q ,m+1}^m}{xy} > \frac{ q ^2 m(2 q -1)}{x}.\end{align*}
Additionally, direct verification shows that
\begin{align*}%\label{eq:positive-terms}
\frac{(2-q)q(m+1)(y-x)}{xy(x-1)(y-1)} + \frac{q^2(3q - mq + 1)}{xy} &> 0 \quad \text{and}\quad 
\frac{2q}{y(x-1)} - \frac{q^2}{yx(x-1)} > 0.
\end{align*} 
Combining these bounds, we obtain the critical lower bound of the target expression, 
\begin{align*}
& \frac{(q+1)(x-y)(2xy - q(x+y-1))}{xy(x-1)(y-1)} 
- \frac{qm(q+1)(2q-1)(x-y)}{xy} \nonumber \\
&\quad + \frac{q(q+1)}{y} 
- \frac{2q^2(q+1)}{x} 
+ \frac{(2-q)q(m+1)(y-x)}{xy(x-1)(y-1)} \nonumber \\
&\quad + \frac{2q}{y(x-1)} 
- \frac{q^2}{yx(x-1)} 
+ \frac{q^2(3q - mq + 1)}{xy} \nonumber \\
>\quad & \frac{q^2 m (2q-1)}{x} - \frac{2(q+1)}{x} - \frac{2q^2(q+1)}{x} > 0
\end{align*}
for all \(m \geq 2\) and \(q \geq 1\). Therefore, \(d_{q,m+1} - d_{q,m} > 0\) holds universally.%, accounting for previously verified exceptional cases.
\end{proof}

\section{Local linearly of $M_\beta(\alpha)$ on $[1-\fr{\beta},1)$.}\label{sec:matching-2}
This section examines the case $\alpha \in [1 - \langle\beta\rangle, 1)$ for $\beta > \sqrt{2}$. Following the methodology of Section \ref{sec:matching-1}, we focus on multinacci numbers $\beta_{q,m} \in (q, q+1)$ defined as the unique real roots of the polynomials
\begin{equation*}
    P_{q,m}(x) = x^m - q x^{m-1}-\ldots -q%\sum_{k=0}^{m-1} x^k
\end{equation*}
for integers $m \geq 2$ and $q \geq 1$. We establish that $M_{\beta_{q,m}}(\alpha)$ exhibits local linearity for Lebesgue almost every $\alpha \in [1 - \langle\beta_{q,m}\rangle, 1)$. This case presents greater analytical complexity than $\alpha \in [0, 1 - \langle\beta_{q,m}\rangle)$.

\begin{theorem}\label{thm:a-e-linear}
For integers $m \geq 2$ and $q \geq 1$, the function $M_{\beta_{q,m}}(\alpha)$ is locally linear for Lebesgue almost every $\alpha \in [1 - \langle\beta_{q,m}\rangle, 1)$.
\end{theorem}

To prove Theorem \ref{thm:a-e-linear}, we introduce essential lemmas and notation. Recall from \eqref{eq:def-T-geq-1} that for $\alpha \in [1 - \langle\beta\rangle, 1)$, the transformation $T_{\beta,\alpha}$ has exactly $\lfloor \beta \rfloor + 2$ branches corresponding to digits $\{0, 1, \ldots, \lfloor \beta \rfloor + 1\}$. The fundamental digit partition of $[0,1]$ is given by
\begin{align*}
    \Delta(0) &:= \left[0, \frac{1 - \alpha}{\beta}\right), \quad 
    \Delta(\lfloor \beta \rfloor + 1) := \left[\frac{\lfloor \beta \rfloor + 1 - \alpha}{\beta}, 1\right], \\
    \Delta(i) &:= \left[\frac{i - \alpha}{\beta}, \frac{i + 1 - \alpha}{\beta}\right) \quad \text{for } 1 \leq i \leq \lfloor \beta \rfloor.
\end{align*}
Each subinterval $\Delta(j)$ consists of points $x \in [0,1)$ where digit $j$ is selected in the next iteration of $T_{\beta,\alpha}$.

The following lemma establishes a finite-variation property for critical orbit differences under $T_{\beta,\alpha}$ at multinacci parameters. For notational clarity in subsequent analysis, we denote by
\begin{equation}\label{eq:def-delta-diff}
    \delta^k(\beta,\alpha) := T_{\beta,\alpha}^k(1) - T_{\beta,\alpha}^k(0) \quad \text{for } k \geq 1.
\end{equation}
\begin{lemma}\label{le:t-1-t-0-finite}
For all integers $m \geq 2$, $q \geq 1$, and every $\alpha \in[1 - \langle\beta_{q,m}\rangle, 1)$, the critical orbit differences exhibit the following algebraic structure,
\begin{equation}\label{eq:delta-decomposition}
\delta^k(\beta_{q,m},\alpha) = \pm\left(\frac{e_1^k(\beta_{q,m},\alpha)}{\beta_{q,m}} + \frac{e_2^k(\beta_{q,m},\alpha)}{\beta_{q,m}^2} + \cdots + \frac{e_m^k(\beta_{q,m},\alpha)}{\beta_{q,m}^m}\right) \quad \forall~ k \in \mathbb{N},
\end{equation}
where coefficients satisfy $e_j^k(\beta_{q,m},\alpha) \in \{0, q\}$ and $e_1^k(\beta_{q,m},\alpha), \ldots, e_m^k(\beta_{q,m},\alpha) \neq q^m$. Moreover, the cardinality of the differences sequence satisfies that 
\begin{equation*}
    \#\left\{ \delta^k(\beta_{q,m},\alpha) : k \in \mathbb{N},\ \alpha \in [1 - \langle\beta_{q,m}\rangle, 1)\right\} \leq 2^{m+1} - 3.
\end{equation*}
\end{lemma}
\begin{proof}
The bound $|\delta^k(\beta_{q,m},\alpha)| < 1$ combined with the identity $\pm \left( \frac{q}{\beta_{q,m}} + \cdots + \frac{q}{\beta_{q,m}^m} \right) = \pm 1$ ensures $(e_1^k(\beta_{q,m},\alpha), \ldots, e_m^k(\beta_{q,m},\alpha)) \neq (q, \ldots, q)$ for all $k \geq 1$. 

The proof is done by induction on $k$. 
For the base case $k=1$, direct computation yields 
\begin{equation*}
    \delta^1(\beta_{q,m},\alpha) = \beta_{q,m} + \alpha - (q + 1) - \alpha = -\frac{q}{\beta_{q,m}^m}.
\end{equation*}
Assume the inductive hypothesis holds for $k-1 \geq 1$, i.e.,
\begin{equation*}
    \delta^{k-1}(\beta_{q,m},\alpha) = \pm\left(\frac{e_1^{k-1}(\beta_{q,m},\alpha)}{\beta_{q,m}} + \cdots + \frac{e_m^{k-1}(\beta_{q,m},\alpha)}{\beta_{q,m}^m}\right)
\end{equation*}
with $e_j^{k-1}(\beta_{q,m},\alpha) \in \{0, q\}$ and $(e_1^{k-1}(\beta_{q,m},\alpha), \ldots, e_m^{k-1}(\beta_{q,m},\alpha)) \neq (q, \ldots, q)$. 

When $(e_1^{k-1}(\beta_{q,m},\alpha), \ldots, e_m^{k-1}(\beta_{q,m},\alpha)) = (0, \ldots, 0)$, we have $\delta^k(\beta_{q,m},\alpha) = 0$. Otherwise, assume $\delta^{k-1}(\beta_{q,m},\alpha) > 0$ (the $\delta^{k-1} (\beta_{q,m},\alpha)< 0$ case follows by similar way). Denote $\ell \in \{0, 1, \ldots, q+1\}$ by the unique integer satisfying $T_{\beta_{q,m},\alpha}^{k-1}(0) \in \Delta(\ell)$. We analyze two cases of  the orbit differences based on the leading coefficient.

\textbf{Case 1:} $e_1^{k-1}(\beta_{q,m},\alpha)=0$ with $(e_1^{k-1}(\beta_{q,m},\alpha) \cdots e_m^{k-1}(\beta_{q,m},\alpha))  \neq 0^m$. Note that $$0 < \delta^{k-1}(\beta_{q,m},\alpha) \leq \sum_{j=2}^m \frac{q}{\beta_{q,m}^j} < \frac{1}{\beta_{q,m}}.$$ So,  $T_{\beta_{q,m},\alpha}^{k-1}(1)$ lies in either $\Delta(\ell)$ or $\Delta(\ell+1)$. 
 If $T_{\beta_{q,m},\alpha}^{k-1}(1) \in \Delta(\ell )$, we have 
    \begin{align*}
        \delta^k(\beta_{q,m},\alpha) &= \beta_{q,m}\delta^{k-1}(\beta_{q,m},\alpha) \\
        &= \frac{e_2^{k-1}(\beta_{q,m},\alpha)}{\beta_{q,m}} + \frac{e_3^{k-1}(\beta_{q,m},\alpha)}{\beta_{q,m}^2} + \cdots + \frac{e_m^{k-1}(\beta_{q,m},\alpha)}{\beta_{q,m}^{m-1}} \quad (\geq 0).
    \end{align*}
     Otherwise,  $T_{\beta_{q,m},\alpha}^{k-1}(1) \in \Delta(\ell +1)$ implies that 
    \begin{align*}
        \delta^k(\beta_{q,m},\alpha) &=\frac{e_2^{k-1}(\beta_{q,m},\alpha)}{\beta_{ q ,m}}+\ldots +\frac{e_m^{k-1}(\beta_{q,m},\alpha)}{\beta_{ q ,m}^{m-1}}-1\\
        &= -\left(\frac{q - e_2^{k-1}(\beta_{q,m},\alpha)}{\beta_{q,m}} + \cdots + \frac{q - e_m^{k-1}(\beta_{q,m},\alpha)}{\beta_{q,m}^{m-1}} + \frac{q}{\beta_{q,m}^m}\right) \quad (\leq 0).
    \end{align*}

\textbf{Case 2:} $e_1^{k-1}(\beta_{q,m},\alpha)=q$ with $e_1^{k-1}(\beta_{q,m},\alpha) \cdots e_m^{k-1}(\beta_{q,m},\alpha) \neq q^m$.
From the inequality $\frac{q}{\beta_{q,m}} < \delta^{k-1}(\beta_{q,m},\alpha) < 1$, we have $T_{\beta_{q,m},\alpha}^{k-1}(1)$ lies in either $\Delta(\ell + q)$ (requiring $\ell \leq 1$) or $\Delta(\ell + q + 1)$ (requiring $\ell = 0$). 
 If $T_{\beta_{q,m},\alpha}^{k-1}(1) \in \Delta(\ell +q)$, by simply computation, we have  
    \begin{align*}
        \delta^k(\beta_{q,m},\alpha) &= \frac{e_2^{k-1}(\beta_{q,m},\alpha)}{\beta_{q,m}} + \cdots + \frac{e_m^{k-1}(\beta_{q,m},\alpha)}{\beta_{q,m}^{m-1}} \quad (\geq 0).
    \end{align*}
However, if $T_{\beta_{q,m},\alpha}^{k-1}(1) \in \Delta(\ell +q+1)$, we have 
    \begin{align*}
        \delta^k(\beta_{q,m},\alpha) &= -\left(\frac{q - e_2^{k-1}(\beta_{q,m},\alpha)}{\beta_{q,m}} + \cdots + \frac{q - e_m^{k-1}(\beta_{q,m},\alpha)}{\beta_{q,m}^{m-1}} + \frac{q}{\beta_{q,m}^m}\right) \quad (\leq 0).
    \end{align*}

In all cases, the form \eqref{eq:delta-decomposition} holds for $\delta^k(\beta_{q,m},\alpha)$. The cardinality bound $2^{m+1} - 3$ follows from the binary choices (sign and coefficients) while excluding the forbidden states $(q, \ldots, q)$ and $(0, \ldots, 0)=-(0, \ldots, 0)$.
\end{proof}

Fix the parameter pair $(\beta_{q,m},\alpha)$.
To simplify notation, for all $k\geq 1$, we express the orbit difference as
\begin{equation}\label{eq:sim-de}
\begin{split}
\delta^k(\beta_{q,m},\alpha) = &\pm\left(\frac{e_1^k(\beta_{q,m},\alpha)}{\beta_{q,m}} + \frac{e_2^k(\beta_{q,m},\alpha)}{\beta_{q,m}^2} + \cdots + \frac{e_m^k(\beta_{q,m},\alpha)}{\beta_{q,m}^m}\right)\\
%=&\pm\left(\frac{e_1^k(\beta_{q,m},\alpha)}{\beta_{q,m}} + \frac{e_2^k(\beta_{q,m},\alpha)}{\beta_{q,m}^2} + \cdots + \frac{e_m^k(\beta_{q,m},\alpha)}{\beta_{q,m}^m}\right)
 =:& \pm e_1^k(\beta_{q,m},\alpha)\cdots e_m^k(\beta_{q,m},\alpha).
 \end{split}\end{equation}
When no ambiguity arises, we may omit the superscript $(\beta_{q,m},\alpha)$.

From Lemma \ref{le:t-1-t-0-finite}, we derive the following evolution rules.
For $\delta^k (\beta_{q,m},\alpha)= \pm 0e^k_2\cdots e^k_m$ with $T_{\beta_{q,m},\alpha}^k(0) \in \Delta(\ell)$ for some $\ell\in\{0,1,\ldots,q+1\}$,
\begin{equation}\label{eq:diff-1}
\delta^{k+1}(\beta_{q,m},\alpha) = \begin{cases}
\pm e^k_2\cdots e^k_m 0 & \text{if } T_{\beta_{q,m},\alpha}^k(1) \in \Delta(\ell), \\
\mp (q-e^k_2)\cdots (q-e^k_m)q & \text{if } T_{\beta_{q,m},\alpha}^k(1) \in \Delta(\ell \pm 1).
\end{cases}
\end{equation}
For  $\delta^k (\beta_{q,m},\alpha)= \pm q e^k_2\cdots e^k_m$ with $T_{\beta_{q,m},\alpha}^k(0) \in \Delta(\ell)$ for some $\ell\in\{0,1,\ldots,q+1\}$,
\begin{equation}\label{eq:diff-2}
\delta^{k+1}(\beta_{q,m},\alpha) = \begin{cases}
\pm e^k_2\cdots e^k_m 0 & \text{if } T_{\beta_{q,m},\alpha}^k(1) \in \Delta(\ell \pm q), \\
\mp (q-e^k_2)\cdots (q-e^k_m)q & \text{if } T_{\beta_{q,m},\alpha}^k(1) \in \Delta(\ell \pm(q+1)).
\end{cases}
\end{equation}
The sign of $\delta^{k+1}$ remains unchanged when both orbits stay in $\Delta(\ell)$ or shift by $\pm q$ intervals, but reverses for $\pm 1$ or $\pm (q+1)$ interval shifts.

We now characterize matching propagation.
When $\delta^{k} (\beta_{q,m},\alpha)= \pm e^k_1\cdots e^k_{m}$ with $e^k_{m}=q$, the first matching occurs at step $k+m$ ($\delta^{k+m}(\beta_{q,m},\alpha)=0$) if and only if  subsequent differences maintain sign consistency, i.e., 
$$\delta^{k+i}(\beta_{q,m},\alpha)=\pm e^k_{i+1}\ldots e^k_{m}0^i \quad \text{for all } 1\leq i\leq m-1.$$
More generally, for $\delta^{k} (\beta_{q,m},\alpha) = \pm e^k_1\cdots e^k_i 0^{m-i}$ with $e^k_i = q$, matching occurs at step $k+i$ if and only if $\delta^{k+1}(\beta_{q,m},\alpha),\ldots,\delta^{k+i}(\beta_{q,m},\alpha)$ share the same sign.
Then we can deduce that no matching occurs before step $m+1$ due to the initial difference structure
$$
\delta^1(\beta_{q,m},\alpha) = -\frac{q}{\beta_{q,m}^m} \quad \forall~ \alpha \in [q+1-\beta_{q,m}, 1)
$$
and the first potential matching arises at step $m+1$.

Following the approach in \cite{Bruin-Keszthelyi-2022}, we prove Theorem \ref{thm:a-e-linear} by establishing the existence of an interval $I \subset [q+1 - \beta_{q,m}, 1)$ such that whenever $T_{\beta_{q,m},\alpha}^k(0) \in I$ for some $k \in \mathbb{N}$, matching occurs within finit time $k' \geq k$ for all $\alpha \in [q+1 - \beta_{q,m}, 1)$. This constitutes the core technical result.
\begin{proposition}\label{cl:matching-P}
For all integers $q \geq 1$ and $m \geq 2$, let $\alpha \in [q+1-{\beta_{q,m}}, 1)$. Denote by
\[
P :=P(\beta_{q,m},\alpha)= \frac{q - \alpha}{{\beta_{q,m}} - 1}.
\]
Suppose there exists a minimal $k \in \mathbb{N}$ such that $T_{{\beta_{q,m}},\alpha}^k(0) = P$. Then,
\begin{itemize}
  \item[(i)] For $q \geq 2$, matching occurs at some finite time $k' \geq k$.
  \item[(ii)] For $q = 1$, matching occurs at some finite time $k' \geq k$ except at most countably many $\alpha \in [2-{\beta_{1,m}}, 1)$.
\end{itemize}
\end{proposition}
\begin{proof}
Let $k $ be the minimal integer with $T_{{\beta_{q,m}},\alpha}^k (0) = P$. Since $P$ is a fixed point satisfying $T_{{\beta_{q,m}},\alpha}(P) = P$, we denote the discrepancy by $\delta^i := T_{{\beta_{q,m}},\alpha}^i(1) - T_{{\beta_{q,m}},\alpha}^i(0)=T_{{\beta_{q,m}},\alpha}^i(1) - P$ for all $i\geq k $. If $\delta^k  = 0$, matching occurs immediately at $k' = k $. Otherwise, Lemma \ref{le:t-1-t-0-finite} implies $|\delta^k | \geq \frac{q}{{\beta_{q,m}}^m} > 0$.

For $\alpha \in [\frac{q}{\beta_{q,m}^m}, \frac{q}{\beta_{q,m}^{m-1}})$, we have that $P + \frac{q}{\beta_{q,m}^m}> 1 $. The orbit position implies 
$$T_{\beta_{q,m},\alpha}^{k+i}(1)\leq  P = T_{\beta_{q,m},\alpha}^{k+i}(0) \quad \text{ for all } i\geq 0.$$
%Through $m-1$ iterations, $\delta^{k+i} < 0$ for all $1 \leq i \leq m-1$.
Then matching propagation rules \eqref{eq:diff-1} and \eqref{eq:diff-2} ensure matching time $k'\leq k + m$. 

For $\alpha \geq \frac{q}{\beta_{q,m}^{m-1}}$, the symmetry property established in Lemma \ref{le:isomorphic-alpha-alpha} allows restriction to $\alpha \in \left[\frac{q}{\beta_{q,m}^{m-1}}, \frac{q+2-\beta_{q,m}}{2}\right)$. Note that $$\frac{q}{\beta_{q,m}^{m-1}} > \frac{q+2-\beta_{q,m}}{2}$$ for $m=2$, $q\geq 1$. Then we only need to consider  $m \geq 3$.
Without loss of generality, assume $\delta^k < 0$. Otherwise, consider the smallest $j \in \{k+1,\dots,k+m-1\}$ with $\delta^j < 0$. If $\delta^{k+i}\geq  0$ for all $1 \leq i \leq m-1$, then matching occurs by step $k'\leq k+m$.
The analysis is divided into the following four cases.\\
\textbf{Case 1:} $\delta^k = -\frac{q}{\beta_{q,m}}$. Direct calculation shows that matching occurring at step $k+1$.\\
\textbf{Case 2:} $\delta^k < -\frac{q}{\beta_{q,m}}$. 
The bounding inequality
\begin{equation*}
    \frac{q - \alpha}{\beta_{q,m}} < P=T_{\beta_{q,m},\alpha}^k(0) < \frac{q + 1 - \alpha}{\beta_{q,m}}
\end{equation*} estimates that 
\begin{equation*}
    T_{\beta_{q,m},\alpha}^k(1)=T_{\beta_{q,m},\alpha}^k(0) +\delta_{k} < P - \frac{q}{\beta_{q,m}} < \frac{1 - \alpha}{\beta_{q,m}}.
\end{equation*}
The recurrence relation yields
\begin{equation}
    T_{\beta_{q,m},\alpha}^{k+1}(1) < T_{\beta_{q,m},\alpha}^{k+1}(0) = P\quad \text{and}\quad  \delta^{k+1} = \beta_{q,m}\delta^k + q < 0.
\end{equation}
\noindent
Through an analysis of consecutive iterations, the following dynamical principles emerge.
 If the inequality $\delta^{k+i} \leq  -\frac{q}{\beta_{q,m}}$ persists for all indices $0 \leq i \leq j\leq m-1$  and  $\delta^{k+j} = -\frac{q}{\beta_{q,m}}$, triggering matching at step $k+j+1$. Conversely, failure of this persistent inequality implies the existence of at least one iteration $2\leq j\leq m-1$ satisfying $0 > \delta^{k+i} > -\frac{q}{\beta_{q,m}}$.

The residual scenario corresponds to $0>\delta^k > -\frac{q}{\beta_{q,m}}$. By Lemma \ref{le:t-1-t-0-finite} we have $$\delta^k = -\left(\frac{e^k_2}{\beta_{q,m}^2}+ \cdots+\frac{ e^k_m }{\beta_{q,m}^m}\right)>-\frac{1}{\beta_{q,m}}\geq -\frac{q}{\beta_{q,m}}$$
where $e^k_2\ldots e^k_m\in \{0, q\}^{m-1} \setminus \{0^{m-1}\}.$
Direct inspection reveals the following two critical criteria (see figure \ref{f:P-1}),
\begin{equation}\label{eq:delta-k+1}
\begin{cases}
\delta^{k+1}=-\left(\frac{e^k_2}{\beta_{q,m}}+ \cdots+\frac{ e^k_m }{\beta_{q,m}^{m-1}}\right) < 0  & \text{if}\quad  \frac{q-\alpha}{\beta_{q,m}} \leq T_{\beta_{q,m},\alpha}^k(1) < P, \\
\delta^{k+1}=\frac{q-e^k_2}{\beta_{q,m}}+ \cdots+\frac{ q-e^k_m }{\beta_{q,m}^{m-1}}+\frac{q}{\beta_{q,m}^m} > 0 & \text{if} \quad   P - \frac{1}{\beta_{q,m}}<T_{\beta_{q,m},\alpha}^k(1)<\frac{q-\alpha}{\beta_{q,m}}.
\end{cases}
\end{equation}
\begin{figure}[h]
\begin{tikzpicture}[scale=4.1]
\draw(-.01,0)node[below]{\footnotesize $0$}--(5/9-2/9,0)node[below]{\tiny $\frac{1-\alpha}{\beta}$}--(10/9-2/9,0)node[below]{\tiny $\frac{2-\alpha}{\beta}$}--(1,0);
\draw[dotted](3/4,0)node[below]{\tiny $P$}--(3/4,3/4)--(3/4-5/9,3/4)--(3/4-5/9,0)node[below]{\tiny $P$-$\frac{1}{\beta}$};
\draw(0,-.01)--(0,1)node[left]{\footnotesize $1$};
\draw[thick, green!60!black] %map
(0,0.4)--(5/9-2/9,1)(5/9-2/9,0)--(10/9-2/9,1)(10/9-2/9,0)--(1,0.2);
\draw[dotted](0,1)--(1,1)--(1,0)(5/9-2/9,0)--(5/9-2/9,1)(10/9-2/9,0)--(10/9-2/9,1);
\draw[dotted](0,0)--(1,1);
\end{tikzpicture}
{\caption{Fix $\beta=1.8$ and $\alpha\in[1-\langle\beta\rangle,1)$.}}
\label{f:P-1}
\end{figure}

\textbf{Case 3:} $\frac{q-\alpha}{\beta_{q,m}} < T_{\beta_{q,m},\alpha}^k(1) < P$.  In this case, $\delta^{k+1}=\beta_{ q ,m}\delta^k<0$. We now state the bound for $\delta^k$
\begin{equation}\label{eq:case-3-all}
\delta^k>\frac{ q  -\alpha}{\beta_{ q ,m}}-\frac{ q  -\alpha}{\beta_{ q ,m}-1}
=\frac{\alpha-\frac{ q }{\beta_{ q ,m}^m}}{\beta_{ q ,m}(\beta_{ q ,m}-1)}-\frac{1}{\beta_{ q ,m}}.
\end{equation}
 If the inequality $\delta^{k+i} < -\frac{q}{\beta_{q,m}}$ or \eqref{eq:case-3-all} persists for all indices $2 \leq i \leq j\leq m-1$, and then $\delta^{k+j} = -\frac{q}{\beta_{q,m}}$, triggering matching at step $k+j+1$. Conversely, failure of this persistent inequality implies the existence of at least one iteration $j'\leq m-1$ satisfying  $$\frac{\alpha-\frac{ q }{\beta_{ q ,m}^m}}{\beta_{ q ,m}(\beta_{ q ,m}-1)}-\frac{1}{\beta_{ q ,m}}>\delta^{k+j'} > -\frac{1}{\beta_{q,m}},$$
 which falls into Case 4.

\textbf{Case 4:} $P - \frac{1}{\beta_{q,m}}<T_{\beta_{q,m},\alpha}^k(1)<\frac{q-\alpha}{\beta_{q,m}}$. By the analysis in Case 3, the bound of $\delta^k$ is 
$$\frac{\alpha-\frac{ q }{\beta_{ q ,m}^m}}{\beta_{ q ,m}(\beta_{ q ,m}-1)}-\frac{1}{\beta_{ q ,m}}>\delta^k > -\frac{1}{\beta_{q,m}}.$$
Here $T_{\beta_{ q ,m},\alpha}^{k+1}(1)>P$ and $$\delta^{k+1}=\beta_{ q ,m}\delta^k+1=  \frac{q-e^k_2}{\beta_{q,m}}+ \cdots+\frac{ q-e^k_m }{\beta_{q,m}^{m-1}}+\frac{ q }{\beta_{q,m}^m}>0.$$

Note that $T_{\beta_{ q ,m},\alpha}(1)=\beta_{ q ,m}+\alpha-( q +1)<P$ for all $\alpha\in [0,1/2+q/{\beta_{ q ,m}^m})$.
Then there are two things that can happen for $\delta^{k+2}$ (See figure \ref{f:P-1} for an illustration). 
\begin{enumerate}
    \item If $T_{\beta_{q,m},\alpha}^{k+1}(1) > \frac{q + 1 - \alpha}{\beta_{q,m}}$, then $T_{\beta_{q,m},\alpha}^{k+2}(1) < P$ and $\delta^{k+2}(\beta_{q,m},\alpha) = \beta_{q,m}\delta^{k+1}(\beta_{q,m},\alpha) - 1 < 0$.
    \item Otherwise, $\delta^{k+2}(\beta_{q,m},\alpha) = \beta_{q,m}\delta^{k+1}(\beta_{q,m},\alpha) > 0$.
\end{enumerate}

To consider the situation of $\delta^{k+2}$, we need more analysis of this dynamical system.
Using the definition of  $\beta_{q,m}$ and $q - \frac{q}{\beta_{q,m}^m} = \beta_{q,m} - 1$, we derive that
\begin{equation}\label{eq:case-3}
\begin{split}
&\frac{1}{\beta_{ q ,m}}-\frac{\alpha-\frac{ q }{\beta_{ q ,m}^m}}{\beta_{ q ,m}(\beta_{ q ,m}-1)}
>\frac{ q }{\beta_{ q ,m}^2}+\ldots+\frac{ q }{\beta_{ q ,m}^{m+1}}-\frac{1-\frac{ q }{\beta_{ q ,m}^m}}{2\beta_{ q ,m}(\beta_{ q ,m}-1)}\\
=&\frac{ q }{\beta_{ q ,m}^2}+\ldots
+\frac{ q }{\beta_{ q ,m}^{m+1}}-\left(\frac{1}{2\beta_{ q ,m}^2}+\frac{1}{2\beta_{ q ,m}^3}+\ldots\right)
+\left(\frac{ q }{2\beta_{ q ,m}^{m+2}}+\frac{ q }{2\beta_{ q ,m}^{m+3}}+\ldots\right)\\
>&\frac{ q }{\beta_{ q ,m}^2}+\frac{1}{2\beta_{ q ,m}^2}+\frac{ q }{2\beta_{ q ,m}^3}   \ldots
+\frac{ q }{2\beta_{ q ,m}^{m+1}}-\frac{ q }{2\beta_{ q ,m}^{m+2}}
-\left(\frac{1}{2\beta_{ q ,m}^2}+\ldots+\frac{1}{2\beta_{ q ,m}^{m+1}}\right)\\
=&\frac{ q }{\beta_{ q ,m}^2} + \frac{ q  -1}{2\beta_{ q ,m}^3}+\ldots+\frac{ q  -1}{2\beta_{ q ,m}^m}
+\frac{ q  -1}{2\beta_{ q ,m}^{m+1}}-\frac{ q }{2\beta_{ q ,m}^{m+2}}.
\end{split}
\end{equation} 
For $b_1\ldots b_m\in\{0,q\}^m$, recall from \eqref{eq:sim-de}, we simply denote by
 $$\frac{b_1}{\beta_{ q ,m}}+\ldots +\frac{b_m}{\beta_{ q ,m}^m}=b_1\ldots b_m.$$
Define that
 $$b_1\ldots b_m^-:=\max\{a_1\ldots a_m\in\{0,q\}^m:a_1\ldots a_m<b_1\ldots b_m\}.$$
Then for any $\alpha\in \left[\frac{q}{\beta_{q,m}^{m-1}}, \frac{q+2-\beta_{q,m}}{2}\right)$, there exist a word  $b^*:=0b_2\ldots b_m\in\{0,q\}^m$ such that %$b^*<0b_2\ldots b_m$,
 \begin{equation}\label{eq:bdd}
\begin{split}
&0b_2\ldots b_m^-\leq P-\frac{p-\alpha}{\beta_{q,m}}<0b_2\ldots b_m\quad \text{and}\\
&\frac 1{\beta_{q,m}}-0b_2\ldots b_m<\frac{p+1-\alpha}{\beta_{q,m}}- P\leq \frac 1{\beta_{q,m}}-0b_2\ldots b_m^-.
\end{split}
\end{equation}

For $q\geq 2$, \eqref{eq:case-3} implies that  $0b_2\ldots b_m\geq 0q^20^{m-3}$.  
Thus, we rewrite $\delta^{k}\leq -b^*$ by 
\begin{equation}\label{eq:re-d-k}
\delta^{k}=-\left(\frac{ q }{\beta_{ q ,m}^2}+\ldots+\frac{ q }{\beta_{ q ,m}^{l_1}}+\frac{ q }{\beta_{ q ,m}^{l_2}} +\ldots+\frac{ q }{\beta_{ q ,m}^{l_3}}+\ldots +\frac{ q }{\beta_{ q ,m}^{l_{j-1}}} +\ldots+\frac{ q }{\beta_{ q ,m}^{l_j}}\right)
\end{equation}
 for some odd  $j\geq 1$ with indices  $3\leq l_1<l_2\leq l_3<\ldots\leq l_{j-2} <l_{j-1}\leq l_j\leq m$. 
This gives that
$$\delta^{k+1}=\beta_{q,m}\delta^k+1=\frac{ q }{\beta_{ q ,m}^{l_1}}+\ldots +\frac{ q }{\beta_{ q ,m}^{l_2-2}} +\frac{ q }{\beta_{ q ,m}^{l_3}}+\ldots +\frac{ q }{\beta_{ q ,m}^{l_j}}+\ldots+\frac{ q }{\beta_{ q ,m}^{m}}>0.$$

From \eqref{eq:case-3} and \eqref{eq:bdd}, we have that $\frac{p+1-\alpha}{\beta_{q,m}}- P<0^2q0^{m-3}$.
Note that if $\delta^{k+1}>\frac{p+1-\alpha}{\beta_{q,m}}- P$, then  $T_{\beta_{q,m},\alpha}^{k+2}(1)<P$ and  $\delta^{k+2}=\beta_{q,m} \delta^{k+1}-1<0$. Otherwise $T_{\beta_{q,m},\alpha}^{k+2}(1)>P$ and  $\delta^{k+2}=\beta_{q,m} \delta^{k+1}>0$. 
By direct computation, we have that $$\beta_{q,m}^{l_1-2} \delta^{k+1}=\frac{q}{\beta_{q,m}^2}+\ldots+\frac{q}{\beta_{q,m}^{m-l_1+2}} >\frac{p+1-\alpha}{\beta_{q,m}}- P.$$
Thus there exists  $i\leq l_1-2$ such that $$\beta_{q,m}^{i} \delta^{k+1}=\frac{ q }{\beta_{ q ,m}^{l_1-i}}+\ldots +\frac{ q }{\beta_{ q ,m}^{l_2-2-i}} +\frac{ q }{\beta_{ q ,m}^{l_3-i}}+\ldots +\frac{ q }{\beta_{ q ,m}^{l_j-i}}+\ldots+\frac{ q }{\beta_{ q ,m}^{m-i}} >\frac{p+1-\alpha}{\beta_{q,m}}- P.$$
Consequently, there exists a unique minimal integer $1 \leq l_1' \leq l_1-1$ satisfying $T_{\beta_{q,m},\alpha}^{k+l_1'}(1) < P$ and \begin{equation}\begin{split}
 \delta^{k+l_1'}=&\beta_{q,m}^{l_1'-1} \delta^{k+1}-1\\
 =-&\left(\frac{ q }{\beta_{ q ,m}}+\ldots+\frac{ q }{\beta_{ q ,m}^{l_1-l_1'}}+\frac{ q }{\beta_{ q ,m}^{l_2-l_1'}} +\ldots+\frac{ q }{\beta_{ q ,m}^{l_3-l_1'}}\right.\\
 &\left.+\ldots +\frac{ q }{\beta_{ q ,m}^{l_{j-1}-l_1'}} +\ldots+\frac{ q }{\beta_{ q ,m}^{l_j-l_1'}}+\frac{q}{\beta_{q,m}^{m-l_1'+2}}+\ldots+\frac{q}{\beta_{q,m}^{m}}\right)<0.\end{split}\end{equation}
The steps $k+l_1'+1$ to $k+l_1$ follow Case 2, while steps $k+l_1$ to $k+l_2-2$ follow Case 3 (if $l_1 > l_2-2$, no Case 3 occurs). Thus, we have 
 \begin{equation} 
 \begin{split}%\frac{ q }{\beta_{ q ,m}^2}+\ldots+\frac{ q }{\beta_{ q ,m}^{l_1}}+
\delta^{k+l_2-2}=&-\left(\frac{ q }{\beta_{ q ,m}^{2}} +\ldots+\frac{ q }{\beta_{ q ,m}^{l_3-(l_2-2)}}+\ldots +\frac{ q }{\beta_{ q ,m}^{l_{j-1}-(l_2-2)}} +\ldots+\frac{ q }{\beta_{ q ,m}^{l_j-(l_2-2)}}\right.\\
&\left. +\frac{ q }{\beta_{ q ,m}^{m-l_2+3}}+\ldots +\frac{ q }{\beta_{ q ,m}^{m+l_1'-l_2+2}}\right).
\end{split}\end{equation}
If $\delta^{k+l_2-2}$ enters Case 3, then we have that 
\begin{equation} 
 \begin{split}
 \delta^{k+l_2-1}%&=\beta_{q,m}\delta^{k+l_2-2}\\
 =&-\left(\frac{ q }{\beta_{ q ,m}} +\ldots+\frac{ q }{\beta_{ q ,m}^{l_3-(l_2-2)-1}}+\ldots +\frac{ q }{\beta_{ q ,m}^{l_{j-1}-(l_2-2)-1}} +\ldots+\frac{ q }{\beta_{ q ,m}^{l_j-(l_2-2)-1}}\right.\\
 &\left. +\frac{ q }{\beta_{ q ,m}^{m-l_2+3-1}}+\ldots +\frac{ q }{\beta_{ q ,m}^{m+l_1'-l_2+2-1}}\right),
\end{split}\end{equation}
This remains in Case 2 until step $k+l_3$. Then if enters Case 3 until step $k+l_4-2$, we set $l_2' = l_4$ and conclude that 
 \begin{equation} 
 \begin{split}%\frac{ q }{\beta_{ q ,m}^2}+\ldots+\frac{ q }{\beta_{ q ,m}^{l_1}}+
\delta^{k+l_4-2}=&-\left(\frac{ q }{\beta_{ q ,m}^{2}} +\ldots+\frac{ q }{\beta_{ q ,m}^{l_5-(l_4-2)}}+\ldots +\frac{ q }{\beta_{ q ,m}^{l_{j-1}-(l_4-2)}} +\ldots+\frac{ q }{\beta_{ q ,m}^{l_j-(l_4-2)}}\right.\\
&\left. +\frac{ q }{\beta_{ q ,m}^{m-l_4+3}}+\ldots +\frac{ q }{\beta_{ q ,m}^{m+l_1'-l_4+2}}\right).
\end{split}\end{equation}
We take $l_2'=l_4$.

Otherwise, $\delta^{k+l_2-2}$ enters Case 4, we have that $\delta^{k+l_2-1} = \beta_{q,m}\delta^{k+l_2-2} + 1 > 0$ and set $l_2' = l_2$. Then there exists $l_2' \leq l_3' \leq l_3-1$ such that $T_{\beta_{q,m}}^{k+l_3'-1} > \frac{q+1-\alpha}{\beta_{q,m}}$ and $\delta^{k+l_3'} < 0$. This enters Case 3 until step $k+l_3$, and then Case 2 until step $k+l_4-2$. So we have that 
 \begin{equation} 
 \begin{split}%\frac{ q }{\beta_{ q ,m}^2}+\ldots+\frac{ q }{\beta_{ q ,m}^{l_1}}+
\delta^{k+l_4-2}=&-\left(\frac{ q }{\beta_{ q ,m}^{2}} +\ldots+\frac{ q }{\beta_{ q ,m}^{l_5-(l_4-2)}}+\ldots +\frac{ q }{\beta_{ q ,m}^{l_{j-1}-(l_4-2)}} +\ldots+\frac{ q }{\beta_{ q ,m}^{l_j-(l_4-2)}}\right.\\
&\left. +\frac{ q }{\beta_{ q ,m}^{m-l_4+3}}+\ldots +\frac{ q }{\beta_{ q ,m}^{m+l_1'-l_4+2}}%
+\frac{ q }{\beta_{ q ,m}^{m+l_2-l_4+3}}+\ldots +\frac{ q }{\beta_{ q ,m}^{m+l_3'-l_4+2}}\right).
\end{split}\end{equation}
After $m$ steps, there is an  increasing sequence $\{i_2,i_4,\ldots,i_{j_1-1}\}\subset \{2,4,\ldots,j-1\}$ such that $l_0'=2$ and $l_n'=l_{i_n}$ for all $n=2,\ldots, j_1-1$ and 
 $$\delta^{k+m}=-\left(\frac{ q }{\beta_{ q ,m}^2}+\ldots+\frac{ q }{\beta_{ q ,m}^{l_1'}}+\frac{ q }{\beta_{ q ,m}^{l_2'}} +\ldots+\frac{ q }{\beta_{ q ,m}^{l_3'}}+\ldots \frac{ q }{\beta_{ q ,m}^{l_{j-1}'}} +\ldots+\frac{ q }{\beta_{ q ,m}^{l_{j_1}'}}\right).$$
By above process, we have that   %$l_i'-l_{i-1}'\leq l_{i_1}-l_{i_1-1}-1
 $l_{i_n}\leq l_{n+1}'\leq l_{i_n+1}-1$ for $n=0,2,\ldots, j_1$ and $j_1\leq j$.
 Iterating this process,  matching occurs within $k+m(m-1)$ steps.

For $q=1$, we consider the matching properties in three cases based on the value of $P - \frac{1-\alpha}{\beta_{1,m}}$. 

For $P - \frac{1-\alpha}{\beta_{1,m}} \geq \frac{1}{\beta_{1,m}^2} + \frac{1}{\beta_{1,m}^3}$, rewrite $\delta^k$ as the form in \eqref{eq:re-d-k}.
Using the same iterative argument as in the $q \geq 2$ case, we prove matching occurs within finite steps. 

If  $\frac{1}{\beta_{1,m}^2} + \sum_{i=4}^m\frac{1}{\beta_{1,m}^i}<P - \frac{1-\alpha}{\beta_{1,m}} < \frac{1}{\beta_{1,m}^2} + \frac{1}{\beta_{1,m}^3}$, we have  $b^* = \frac{1}{\beta_{1,m}^2} + \frac{1}{\beta_{1,m}^3}$ which satisfies \eqref{eq:bdd}. For $\delta^k \leq  -b^*$ with $m \geq 4$, the $q \geq 2$ argument applies. The case $m=3$ and $\delta^k = -b^*$  requires separate consideration. 
If $\frac{2-\alpha}{\beta_{1,3}} - P < \frac{1}{\beta_{1,3}^3}$, then
  \[
  \delta^{k+1} = \frac{1}{\beta_{1,3}^3}, \quad \delta^{k+2} = -\left( \frac{1}{\beta_{1,3}} + \frac{1}{\beta_{1,3}^3} \right), \quad \delta^{k+3} = -\frac{1}{\beta_{1,3}^2}, \quad \delta^{k+4} = -\frac{1}{\beta_{1,3}}
  \]
  and matching occurs at step $k+5$.  
But if $\frac{1}{\beta_{1,3}^3} < \frac{2-\alpha}{\beta_{1,3}} - P < \frac{1}{\beta_{1,3}^2}$, then
  \[
  \delta^{k+1} = \frac{1}{\beta_{1,3}^3}, \quad \delta^{k+2} = \frac{1}{\beta_{1,3}^2}, \quad \delta^{k+3} = -\left( \frac{1}{\beta_{1,3}^2} + \frac{1}{\beta_{1,3}^3} \right) = \delta^k
  \]
  resulting in a $3$-periodic orbit. In this case, the intermediate $\beta$-expansions of $0$ and $1$ develop periodic tails, respectively,
\begin{itemize}
  \item the coding of $0$ is $b_1\ldots b_{k-1}1^\infty$;
  \item the coding of $1$ is $a_1\ldots a_{k-1}(012)^\infty$.
\end{itemize}
By Claim \ref{cl:N_k-N} in the proof of Lemma \ref{lem:right-continuity-M}, for each fixed prefix pair $(b_1\ldots b_{k-1}, a_1\ldots a_{k-1})$, at most one $\alpha$ satisfies both the inequality condition and the periodic coding requirement. Consequently, the set of such exceptional $\alpha$ is at most countable.

For $\frac{1}{\beta_{1,m}^2} \leq P - \frac{1-\alpha}{\beta_{1,m}} \leq \frac{1}{\beta_{1,m}^2} + \frac{1}{\beta_{1,m}^3}$,  
here $b^* = 010e_4\ldots e_m$ with $e_i \in \{0,1\}$ and $\frac{2-\alpha}{\beta_{1,m}} - P > 001(1-e_4)\ldots(1-e_m)1$. Express $\delta^k<-b^*$ as
\begin{equation}\label{eq:re-d-k-1}
\delta^{k} = -\left( \frac{1}{\beta_{1,m}^2} + \cdots + \frac{1}{\beta_{1,m}^{l_1}} + \frac{1}{\beta_{1,m}^{l_2}} + \cdots + \frac{1}{\beta_{1,m}^{l_j}} \right)
\end{equation}
for odd $j \geq 1$ and indices $2 \leq l_1 < l_2 \leq \cdots \leq l_j \leq m$. Then,
\[
\delta^{k+1} = \beta_{1,m}\delta^k + 1 = \frac{1}{\beta_{1,m}^{l_1}} + \cdots + \frac{1}{\beta_{1,m}^{l_2-2}} + \frac{1}{\beta_{1,m}^{l_3}} + \cdots + \frac{1}{\beta_{1,m}^m} > 0.
\]
Compare $\beta_{1,m}^{l_1-3} \delta^{k+1}$ to $\frac{2-\alpha}{\beta_{1,m}} - P$.   
If $\beta_{1,m}^{l_1-3} \delta^{k+1} > \frac{2-\alpha}{\beta_{1,m}} - P$, then we have 
  \[
  \delta^{k+l_1-2} = -\left( \frac{1}{\beta_{1,m}} + \frac{1}{\beta_{1,m}^2} + \frac{1}{\beta_{1,m}^{l_2-l_1+2}} + \cdots + \frac{1}{\beta_{1,m}^m} \right).
  \]
 Otherwise, $\beta_{1,m}^{l_1-2} \delta^{k+1} > \frac{2-\alpha}{\beta_{1,m}} - P$ implies that
  \[
  \delta^{k+l_1-1} = -\left( \frac{1}{\beta_{1,m}} + \frac{1}{\beta_{1,m}^{l_2-l_1+1}} + \cdots + \frac{1}{\beta_{1,m}^m} \right).
  \]

After $m$ steps, there exist indices $\{i_2,\ldots,i_{j_1-1}\} \subset \{2,4,\ldots,j-1\}$ with $l_0' = 2$, $l_n' = l_{i_n}$, and $l_{i_n} \leq l_{n+1}' \leq l_{i_n+1}-1$ for $n=0,2,\ldots,j_1$ ($j_1 \leq j$) such that:
\[
\delta^{k+m} = -\left( \frac{1}{\beta_{1,m}^2} + \cdots + \frac{1}{\beta_{1,m}^{l_1'}} + \frac{1}{\beta_{1,m}^{l_2'}} + \cdots + \frac{1}{\beta_{1,m}^{l_{j_1}'}} \right).
\]

The iteration process leads to two distinct outcomes. The sequence $\{\delta^n\}$ eventually satisfies the matching condition within finite steps.
 Otherwise, after $\kappa m$ iterations, the difference term becomes that 
  \[
  \delta^{k+\kappa m} = -\left( \frac{1}{\beta_{1,m}^2} + \frac{1}{\beta_{1,m}^{l_1}} + \cdots + \frac{1}{\beta_{1,m}^{l_{j'}}} \right) < \frac{1-\alpha}{\beta_{1,m}} - P,
  \]
  with $\delta^{k+\kappa m + l_i - 1} < \frac{1-\alpha}{\beta_{1,m}} - P$ for all $1 \leq i \leq j'$. After $m$ additional steps, it returns to its previous state,
  \[
  \delta^{k+(\kappa+1)m} = -\left( \frac{1}{\beta_{1,m}^2} + \frac{1}{\beta_{1,m}^{l_1}} + \cdots + \frac{1}{\beta_{1,m}^{l_{j'}}} \right) = \delta^{k+\kappa m},
  \]
  establishing a periodic orbit with period $m$. For instance, when $m \geq 4$ is even and $b^* = (01)^{m/2}$, $\delta^k = -b^*$ induces such periodicity.
  
  Each periodic pattern corresponds to at most one $\alpha$ (determined by the $\beta$-expansion condition). Since these patterns are characterized by eventually periodic expansions of $0$ and $1$, and the set of eventually periodic sequences is countable, the set of $\alpha \in [1 - \langle \beta_{1,m} \rangle, 1)$ without matching is at most countable. Consequently, for $q=1$, matching occurs within finite steps except for a countable set of $\alpha$.
\end{proof}
\begin{proof}[Proof of Theorem \ref{thm:a-e-linear}]
By Lemma \ref{le:linea-on-matching-interval}, it suffices to show that for any $m \geq 2$ and $q \geq 1$, the transformation $T_{\beta_{q,m},\alpha}$ exhibits matching for Lebesgue almost every $\alpha \in [1 - \langle\beta_{q,m}\rangle, 1)$.

Let $P = P_{\beta_{q,m},\alpha} = \frac{q - \alpha}{\beta_{q,m} - 1}$. For $q \geq 2$, Proposition \ref{cl:matching-P} implies that matching occurs at step $\ell > k$ when $T_{\beta_{q,m},\alpha}^k(0) = P$. Fix $\varepsilon > 0$ such that for $I = (P - \varepsilon, P)$, if $T_{\beta_{q,m},\alpha}^k(0) \in I$, the sign sequence $\left( T_{\beta_{q,m},\alpha}^i(1) - T_{\beta_{q,m},\alpha}^i(0) \right)_{i=k}^{\ell}$ coincides with $\left( T_{\beta_{q,m},\alpha}^i(1) - P \right)_{i=k}^{\ell}$, ensuring matching.

For $q = 1$, the preceding proof establishes that matching fails for at most countably many $\alpha \in [1 - \langle\beta_{1,m}\rangle, 1)$. 

By Theorem 3.1 in \cite{Bruin-Keszthelyi-2022}, the orbit $\{T_{\beta_{q,m},\alpha}^k(0)\}$ is dense in $[0,1)$ for Lebesgue almost every $\alpha \in [q+1 - \beta_{q,m}, 1)$. Therefore, for $q \geq 2$, $T_{\beta_{q,m},\alpha}^k(0) \in I$ occurs almost surely, triggering matching. 
    When $q = 1$, the countable exclusion set has Lebesgue measure zero.
Since $[q+1 - \beta_{q,m}, 1) = [1 - \langle\beta_{q,m}\rangle, 1)$, matching holds for Lebesgue almost every $\alpha$ in this interval.
\end{proof}

\begin{remark}
Generally, it is not easy to find an $\alpha$ such that $T_{\beta_{ q ,m},\alpha}$ has no matching. But for $q=1$, the proof  of Proposition \ref{cl:matching-P} provides a way to find such an $\alpha$.
\end{remark}

\begin{figure}[h!]
  \centering
   \includegraphics[width=5.5cm]{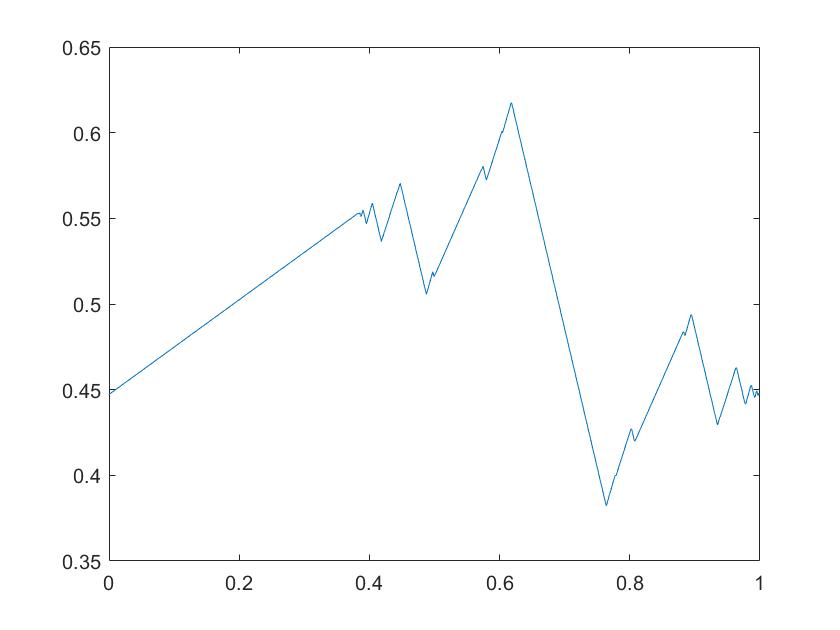}\quad \includegraphics[width=5.5cm]{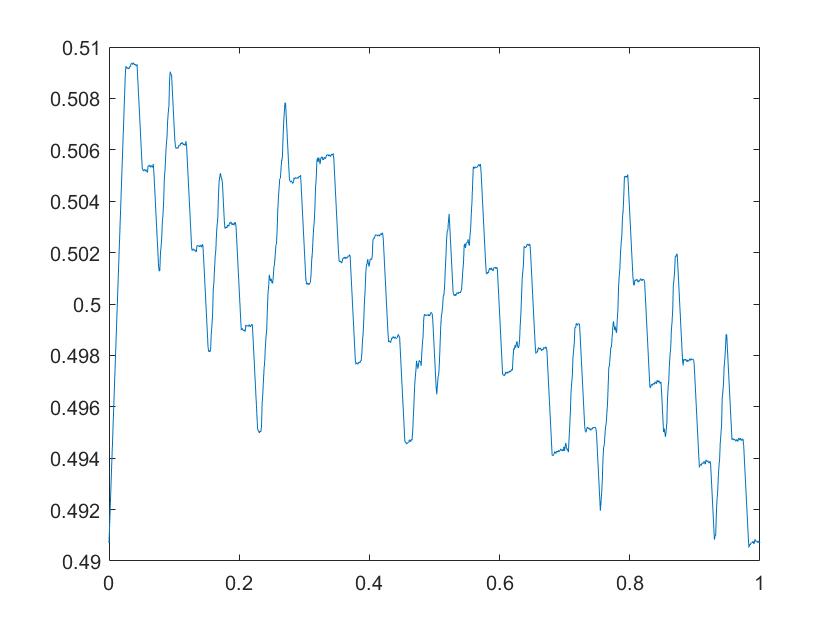}\\
  \caption{Left: $M_{\beta_{1,2}}(\alpha)$ for $\alpha\in[0,1]$ where $\beta_{1,2}$ is the gold mean. Right: $M_{\beta_{2,4}}(\alpha)$ for $\alpha\in[0,1]$ where $\beta_{2,4}=2.974449244\ldots $.}\label{Fig:local-linear}
\end{figure}
Given $m \geq 2$ and $q \geq 1$, we have shown that $[0, 1-\langle\beta_{q,m}\rangle)$ is a matching interval where $M_{\beta_{q,m}}(\alpha)$ increases linearly. For Lebesgue almost every $\alpha \in [ 1-\langle\beta_{q,m}\rangle, 1)$, $M_{\beta_{q,m}}(\alpha)$ is locally linear (see Figure \ref{Fig:local-linear}). 
These observations raise a natural question: on which matching intervals is $M_{\beta_{q,m}}(\alpha)$ increasing, and on which is it decreasing? 

Let $I_{\beta_{q,m}}(a_1\ldots a_{k+1}, b_1\ldots b_{k+1}) \subset [1-\langle\beta_{q,m}\rangle, 1)$ be a matching interval for $k \geq m$. By \eqref{eq:la-i} in the proof of Lemma \ref{le:linea-on-matching-interval}, we have that 
 $\delta^i(\beta_{q,m},\alpha)$ is constant on the interval for $1 \leq i \leq k$ and  $\delta^{k+1}(\beta_{q,m},\alpha) = 0$.
We state the monotonicity of $M_{\beta_{q,m}}$ depends on the sign of $\delta^k(\beta_{q,m},\alpha)$.

\begin{proposition}\label{prop:mono}
Given $m \geq 2$ and $q \geq 1$, let $I_{\beta_{q,m}}(a_1\ldots a_{k+1}, b_1\ldots b_{k+1}) \subset [1-\langle\beta_{q,m}\rangle, 1)$ be a matching interval for $k \geq m$. Then $M_{\beta_{q,m}}(\alpha)$ is increasing on this interval if $\delta^k(\beta_{q,m},\alpha) > 0$. Otherwise, $M_{\beta_{ q ,m}}(\alpha)$ is decreasing.
\end{proposition}
\begin{proof}
Take $\alpha \in I_{\beta_{ q ,m}}(a_1\ldots  a_{k+1},b_1\ldots  b_{k+1})$. We previously established in Lemma~\ref{le:linea-on-matching-interval} that
\begin{align*}
M_{\beta_{q,m}}(\alpha) 
= \frac{1}{\sum_{i=0}^{k} \lambda_i} 
\sum_{i=0}^{k} \frac{\big(T_{\beta_{q,m},\alpha}^i(1) - T_{\beta_{q,m},\alpha}^i(0)\big) g_i(\alpha)}{\beta_{q,m}^i}
\end{align*}
where 
\[
g_i(\alpha) = \dfrac{\beta_{q,m}^i}{2} \left( 
1 - \sum_{j=1}^{i} \frac{a_j + b_j}{\beta_{q,m}^j} 
+ 2\alpha \sum_{j=1}^{i} \frac{1}{\beta_{q,m}^j} 
\right).
\]
Since $\sum_{i=0}^{k} \lambda_i \geq 0$, for the increasing of $M_{\beta_{q,m}}$ it suffices to show 
\begin{equation}\label{eq:coeffi}
C^k(\beta_{q,m},\alpha) := \sum_{i=1}^{k} \frac{(\beta_{q,m}^{i-1} + \cdots + 1) \big( T_{\beta_{q,m},\alpha}^i(1) - T_{\beta_{q,m},\alpha}^i(0) \big)}{\beta_{q,m}^i}
\end{equation}
is positive when $\delta^k(\beta_{q,m},\alpha) > 0$, and negative when $\delta^k < 0$ for decreasing $M_{\beta_{q,m}}$.

We establish the sign relationships for matching at step $k+1$, for which $\delta^{k-m+1}(\beta_{q,m},\alpha)$ through $\delta^k(\beta_{q,m},\alpha)$ share the same sign, while $\delta^{k-m}(\beta_{q,m},\alpha)$ exhibits the opposite sign. Note this sign alternation is necessary; otherwise, matching would occur earlier at step $k+1$.

%$$-\frac{q}{\beta_{q,m}^2},\quad \frac{q}{\beta_{q,m}^2},\quad \ldots, \quad -\frac{q}{\beta_{q,m}^2},\quad \frac{q}{\beta_{q,m}^2},\quad \frac{q}{\beta_{q,m}}.$$
For the specific case $m=2$, % the difference sequence follows the alternating pattern.
if matching occurs in even  steps $(k+1) \geq 4$, the difference sequence follows the alternating pattern
    \[
    \delta^i(\beta_{q,2},\alpha) = \begin{cases} 
        -\frac{q}{\beta_{q,2}^2} & \text{for odd } i\leq k-1, \\
       \frac{q}{\beta_{q,2}^2} & \text{for even } i\leq k-1,
        %q\beta_{q,2}^{-1} & \text{at } i = k
    \end{cases}
    \]
    and then $\delta^k(\beta_{q,2},\alpha)=\frac{q}{\beta_{q,2}}$. This 
    yields the positive coefficient expression
\begin{align*}
C^k(\beta_{q,2},\alpha)=\sum_{i=1}^{k} \frac{(\beta_{q,2}^{i-1} + \cdots + 1) \delta^i(\beta_{q,2},\alpha)}{\beta_{q,2}^i}=\sum_{i=2}^{(k+1)/2} \frac{q}{\beta_{q,2}^{2i}}+\frac{(\beta_{q,2}^{k-1} + \cdots + 1)}{\beta_{q,2}^k}\frac{q}{\beta_{q,2}}.
\end{align*}

When matching occurs in odd  steps $(k+1) \geq 3$, the difference sequence are $\delta^k(\beta_{q,2},\alpha)=-\frac{q}{\beta_{q,2}}$ and 
    \[
    \delta^i(\beta_{q,2},\alpha) = \begin{cases} 
        -\frac{q}{\beta_{q,2}^2} & \text{for odd } i\leq k-1, \\
       \frac{q}{\beta_{q,2}^2} & \text{for even } i\leq k-1,
        %q\beta_{q,2}^{-1} & \text{at } i = k
    \end{cases}
    \]
    producing the negative coefficient
    \begin{align*}
C^k(\beta_{q,2},\alpha)
=\sum_{i=2}^{k/2 }\frac{q}{\beta_{q,2}^{2i}}-\frac{(\beta_{q,2}^{k-2} + \cdots + 1)}{\beta_{q,2}^{k-1}}\frac{q}{\beta_{q,2}^2}-\frac{(\beta_{q,2}^{k-1} + \cdots + 1)}{\beta_{q,2}^{k}}\frac{q}{\beta_{q,2}}.\end{align*}

Now consider $m \geq 3$. For $k=m$, we have $\delta^1(\beta_{q,m},\alpha) = -\frac{q}{\beta_{q,m}^{-m}}$, and all differences $\delta^i(\beta_{q,m},\alpha)$ for $1 \leq i \leq m$ remain negative. Consequently, by \eqref{eq:coeffi},
\begin{equation}%\label{eq:coeffi_neg}
C^m(\beta_{q,m},\alpha) = \sum_{i=1}^{m} \frac{(\beta_{q,m}^{i-1} + \cdots + 1) \delta^i(\beta_{q,m},\alpha)}{\beta_{q,m}^i} < 0,
\end{equation}
establishing that $M_{\beta_{q,m}}(\alpha)$ decreases strictly on those matching intervals with matching at step $m+1$.

For $k = m+1$, we have $\delta^1(\beta_{q,m},\alpha) = -\frac{q}{\beta_{q,m}^m}$ and
\[
\delta^2(\beta_{q,m},\alpha) = q^{m-2}0q, \quad \delta^3(\beta_{q,m},\alpha) = q^{m-3}0q0, \quad \ldots, \quad \delta^{m-1}(\beta_{q,m},\alpha) = q0^{m-1},
\]
all of which are positive. Therefore,
\begin{equation}%\label{eq:coeffi_pos}
C^{m+1}(\beta_{q,m},\alpha) = \sum_{i=1}^{m+1} \frac{(\beta_{q,m}^{i-1} + \cdots + 1) \delta^i(\beta_{q,m},\alpha)}{\beta_{q,m}^i} > 0.
\end{equation}
Thus, $M_{\beta_{q,m}}(\alpha)$ is strictly increasing on intervals where matching occurs at step $m+2$.

For $k = m+2$, the monotonicity of $M_{\beta_{q,m}}(\alpha)$ depends on the sequences $\delta^1(\beta_{q,m},\alpha), \dots, \delta^{m+2}(\beta_{q,m},\alpha)$, which can be either:
\begin{align*}
\text{Case 1:} \quad & -0^{m-1}q,\ -0^{m-2}q0,\quad q^{m-3}0q^2,\ \dots,\ q0^{m-1} \quad \text{or} \\
\text{Case 2:} \quad & -0^{m-1}q,\ q^{m-2}0q,\quad -0^{m-3}q0q,\ \dots,\ -q0^{m-1}.
\end{align*}
The corresponding coefficient sum is given by
\begin{equation}%\label{eq:coeff_sum}
C^{m+2}(\beta_{q,m},\alpha) = \sum_{i=1}^{m+2} \frac{(\beta_{q,m}^i - 1) \delta^i(\beta_{q,m},\alpha)}{\beta_{q,m}^i (\beta_{q,m} - 1)}.
\end{equation}
A direct calculation shows that $C^{m+2} > 0$ in Case 1 and $C^{m+2} < 0$ in Case 2.

For $k = m+3$, the sequences $\delta^1(\beta_{q,m},\alpha), \dots, \delta^{m+3}(\beta_{q,m},\alpha)$ admit one of four cases, which are
\begin{align*}
\text{Case A:} \quad & -0^{m-1}q,\ -0^{m-2}q0,\ -0^{m-3}q0^2,\quad q^{m-4}0q^3,\ \dots,\ q0^{m-1}; \\
\text{Case B:} \quad & -0^{m-1}q,\ -0^{m-2}q0,\ q^{m-3}0q^2,\quad -0^{m-4}q0^2q,\ \dots,\ -q0^{m-1}; \\
\text{Case C:} \quad & -0^{m-1}q,\ q^{m-2}0q,\ q^{m-3}0q0,\quad -0^{m-4}q0q^2,\ \dots,\ -q0^{m-1}; \\
\text{Case D:} \quad & -0^{m-1}q,\ q^{m-2}0q,\ -0^{m-3}q0q,\quad q^{m-4}0q0q,\ \dots,\ q0^{m-1}.
\end{align*}
Using the same coefficient sum formula
\[
C^{m+3}(\beta_{q,m},\alpha) = \sum_{i=1}^{m+3} \frac{(\beta_{q,m}^i - 1) \delta^i(\beta_{q,m},\alpha)}{\beta_{q,m}^i (\beta_{q,m} - 1)},
\]
we find $C^{m+3} > 0$ in Cases B and C, while $C^{m+3} < 0$ in Cases A and D. Note that Case A cannot occur when $m=3$.

Now suppose that for all matching interval  $I_{\beta_{q,m}}(a_1\ldots a_{j}, b_1\ldots b_{j}) \subset [1-\langle\beta_{q,m}\rangle, 1)$ with $m+3 \leq j \leq k$, the function $M_{\beta_{q,m}}(\alpha)$ is increasing whenever $\delta^{j-1}(\beta_{q,m},\alpha) > 0$ and decreasing whenever $\delta^{j-1}(\beta_{q,m},\alpha) < 0$. 

Consider a matching interval  $I_{\beta_{q,m}}(a_1\ldots a_{k+1}, b_1\ldots b_{k+1}) \subset [1-\langle\beta_{q,m}\rangle, 1)$. For the initial segment, assume 
\[
\delta^{k-(m+1)}(\beta_{q,m},\alpha) = -\left( \frac{e_1}{\beta_{q,m}} + \cdots + \frac{e_m}{\beta_{q,m}^m} \right), \quad \text{where } e_1 \ldots e_m \in \{0,q\}^m \setminus \{0^m,q^m\}.
\]
The sequence $\{\delta^\ell(\beta_{q,m},\alpha)\}_{k-(m+1) \leq \ell \leq k}$ must be one of the following cases,
\begin{equation}\label{eq:matching-k}
\begin{split}
\text{Case 1:} \quad & -e_1\ldots e_m,\  -e_2\ldots e_m0,\quad   (q-e_3)\ldots (q-e_m)q^2,\ \ldots,\ q0^{m-1} \\
\text{Case 2:} \quad & -e_1\ldots e_m,\  (q-e_2)\ldots (q-e_m)q,\quad   -e_3\ldots e_m 0q,\ \ldots,\ -q0^{m-1}.
\end{split}
\end{equation}

Let $I_k$ be a matching interval with matching at step $k$. For all $\alpha' \in I_k$, the first $k-(m+1)$ terms match, and $\{\delta^\ell(\beta_{q,m},\alpha')\}_{k-(m+1) \leq \ell \leq k-1}$ is 
\[
-e_1\ldots e_m,\  (q-e_2)\ldots (q-e_m)q,\  (q-e_3)\ldots (q-e_m)q0,\  \ldots,\ q0^{m-1}.
\]
The coefficient sum satisfies
\begin{align*}
C^{k}(\beta_{q,m},\alpha') 
&= \sum_{\ell=1}^{k-m-1} \frac{(\beta_{q,m}^{\ell-1} + \cdots + 1) \delta^\ell(\beta_{q,m},\alpha')}{\beta_{q,m}^\ell} \\
&\quad + \sum_{\ell=k-m}^{k-1} \frac{(\beta_{q,m}^{\ell-1} + \cdots + 1) \delta^\ell(\beta_{q,m},\alpha')}{\beta_{q,m}^\ell} > 0.
\end{align*}

For Case 1 in \eqref{eq:matching-k}, since the initial $k-(m+1)$ terms of $C^{k+1}(\beta_{q,m},\alpha) $ and $C^{k}(\beta_{q,m},\alpha') $ are identical, it suffices to compare the remaining segments. By direct term-wise comparison of the sequences, we obtain that 
\begin{align*}
& \sum_{\ell=k-m}^{k} \frac{(\beta_{q,m}^{\ell-1} + \cdots + 1) \delta^\ell(\beta_{q,m},\alpha)}{\beta_{q,m}^\ell} - \sum_{\ell=k-m}^{k-1} \frac{(\beta_{q,m}^{\ell-1} + \cdots + 1) \delta^\ell(\beta_{q,m},\alpha')}{\beta_{q,m}^\ell} \\
&= -\frac{\beta_{q,m}^{k-m-1} + \cdots + 1}{\beta_{q,m}^{k-m}} + \sum_{\ell=k-m+1}^{k} \frac{(\beta_{q,m}^{\ell-1} + \cdots + 1) q}{\beta_{q,m}^{k+1}} \\
&> -\frac{\beta_{q,m}^{k-m-1} + \cdots + 1}{\beta_{q,m}^{k-m}} +  \frac{\beta_{q,m}^{k-m} + \cdots + 1}{\beta_{q,m}^{k-m+1} } \sum_{r=1}^{m}\frac{q}{\beta_{q,m}^r} = \frac{1}{\beta_{q,m}^{k-m+1}} > 0,
\end{align*}
where the first equality results from term-by-term comparison of the sequences.
 Thus $M_{\beta_{q,m}}$ is increasing in Case 1.

For Case 2 in \eqref{eq:matching-k}, %:
%\[
%-e_1\ldots e_m,\  (q-e_2)\ldots (q-e_m)q,\  -e_3\ldots e_m 0q,\  \ldots,\  -q0^{m-1},
%\]
there exists a matching interval $I_\kappa$ (matching step $\kappa$ between $k-m$ and $k-1$) sharing the first $k-(m+1)-2$ terms and the coefficient sum is smaller than $0$. Its last $m$ terms are
\[
-e_0e_1\ldots e_{m},\  -e_1\ldots e_{m-1}0,\  \ldots,\  -e_{m}0^{m-1},
\]
where if $e_\ell \ldots e_{m}0^{\ell-1} = 0^m$ for some $\ell$, subsequent terms vanish. %For step $k+1$ with $e_m = 0$, the last $m+3$ terms are
%\[
%-e_0e_1\ldots e_{m-1},\  -e_1\ldots e_{m-1}0,\  (q-e_2)\ldots (q-e_{m-1})q^2,\  -e_3\ldots e_{m-1}0^2q,\  \ldots,\  -q0^{m-1}.
%\]
A similar calculation shows the coefficient sum of $C^{k+1}(\beta_{q,m},\alpha) $  is smaller, confirming the decreasing property.
\end{proof}

 \section{Further questions}\label{sec:question}
We conclude this paper with one open problem.
%\begin{enumerate}
%    \item We established continuity of $M_{\beta}(\alpha)$ on $[0,1]$. What can be said about its differentiability on this interval?
%    
For fixed $\beta = \beta_{q,m}$ with $m \geq 2$ and $q \geq 1$, we proved that $M_{\beta}(\alpha)$ exhibits matching for Lebesgue almost every $\alpha \in [0,1]$. Define
    \[
    \Lambda(\beta_{q,m}) := \left\{ \alpha \in [0,1] : T_{\beta_{q,m},\alpha} \text{ has no matching} \right\}.
    \]
    What is the Hausdorff dimension of $\Lambda(\beta_{q,m})$? While
    \[
    \dim_H (\Lambda(\beta_{q,2})) = \frac{\log q}{\log \beta_{q,2}}
    \]
    was established for $m=2$ and any $q \geq 1$ in \cite{Bruin-Carminati-Kalle-2017}, we conjecture that for general $m \geq 2$,
    \[
    \dim_H (\Lambda(\beta_{q,m})) = \frac{\log q}{\log \beta_{q,m}}.
    \]

\section*{Acknowledgments}
Part of this work was completed during a visit by the second author to Leiden University. The authors thank the Mathematical Institute of Leiden University for its hospitality, and particularly express gratitude to Charlene Kalle. We also thank Derong Kong for valuable suggestions that improved the exposition.

%\bibliographystyle{abbrv}
%\bibliography{Fractal-Expansions}

\end{document}